\documentclass[preprint,10pt]{elsarticle}

\usepackage{graphicx}

\usepackage{amssymb}
\usepackage{amsthm}
\usepackage{mathrsfs}
\usepackage{mathtools}

\newtheorem{theorem}{Theorem}
\newtheorem{example}{Example}
\newtheorem{remark}{Remark}
\newtheorem{proposition}{Proposition}
\newtheorem{definition}{Definition}
\newtheorem{lemma}{Lemma}
\newtheorem{corollary}{Corollary}

\newcommand{\NN}{\mathbb{N}}
\newcommand{\RR}{\mathbb{R}}
\newcommand{\II}{\mathbb{I}}
\newcommand{\QQ}{\mathbb{Q}}

\newcommand{\fd}{\mathbf{fd}}
\newcommand{\EE}{\mathbf{E}}

\newcommand{\A}{\mathcal{A}}

\newcommand{\B}{\mathcal{B}}
\newcommand{\C}{\mathcal{C}}
\newcommand{\D}{\mathcal{D}}
\newcommand{\E}{\mathcal{E}}
\newcommand{\F}{\mathcal{F}}
\renewcommand{\H}{\mathcal{H}}

\newcommand{\W}{\mathcal{W}}
\newcommand{\U}{\mathcal{U}}
\renewcommand{\L}{\mathcal{L}}

\newcommand{\G}{\mathcal{G}}
\renewcommand{\P}{\mathcal{P}}

\newcommand{\Q}{\mathcal{Q}}

\renewcommand{\u}{\underline}
\renewcommand{\o}{\overline}

\renewcommand{\j}{\mathbf{j}}
\newcommand{\dd}{\mathbf{d}}
\newcommand{\p}{\mathbf{p}}

\renewcommand{\d}{\mathrm{d}}
\renewcommand{\sc}{\mathbf{sc}}
\newcommand{\fsc}{\mathbf{fsc}}

\newcommand{\card}{\mathrm{card}\,}

\newcommand{\minimize}{\mathrm{minimize}}

\newcommand{\gn}{\subseteq_{GN}}

\makeatletter
\newcommand{\superimpose}[2]{%
  {\ooalign{$#1\@firstoftwo#2$\cr\hfil$#1\@secondoftwo#2$\hfil\cr}}}
\makeatother

\newcommand{\intC}{\mathpalette\superimpose{{\hspace{0.3mm}\vspace{-0.2mm}\textsc{C}}{\int}}}
\newcommand{\intCs}{\mathpalette\superimpose{{\hspace{0.1mm}\textsc{c}}{\int}}}

\journal{arXiv.org}

\begin{document}

\begin{frontmatter}



\title{Envelopes of conditional probabilities extending a strategy and a prior probability}


\author[label1]{Davide Petturiti\corref{cor1}}
\ead{davide.petturiti@dmi.unipg.it}
\author[label2]{Barbara Vantaggi}
\ead{barbara.vantaggi@sbai.uniroma1.it}
\address[label1]{Dip.  Matematica e Informatica, Universit\`{a} di Perugia, Italy}
\address[label2]{Dip.  S.B.A.I., ``La Sapienza'' Universit\`{a} di Roma, Italy}
\cortext[cor1]{Corresponding author.}

\begin{abstract}
Any strategy and prior probability together are a coherent conditional probability that can be extended, generally not in a unique way, to a full conditional probability. The corresponding class of extensions is studied and a closed form
expression for its envelopes is provided.
Then a topological characterization of the subclasses of extensions satisfying the further properties of full disintegrability and full strong conglomerability is given and their envelopes are studied.
\end{abstract}


\end{frontmatter}

\section{Introduction}
In this paper we provide closed form expressions for the lower and upper envelopes of sets of full conditional probabilities extending a strategy and a prior probability. In the seminal paper \cite{dubins} by Dubins, the notion of {\it strategy} $\sigma$ together with the ensuing concepts of {\it disintegrability} and {\it (strong) conglomerability} with respect to a {\it  (finitely additive) prior probability} $\pi$ are presented and it is proved  that the assessment $\{\pi,\sigma\}$ can always be extended, generally not in a unique way, to a full conditional probability. The notion of coherence, essentially due to de Finetti \cite{definetti,williams,holzer,regazzini}, and the extension of an assessment  $\{\pi,\sigma\}$ are particularly meaningful in Bayesian statistics \cite{definetti,heath1978,regazzini-bayes,regazzini-unimi,brr-bayes,kss-fa-statistics} and also in limit theorems and stochastic processes 
\cite{ds-libro,dubins,dubins2,purves}.

The aforementioned papers generally focus on the existence of a particular conditional probability extending the assessment $\{\pi,\sigma\}$, while here the whole class of extensions is considered.
In detail, we provide a closed form expression for the envelopes
of the whole class of  full conditional probabilities extending a strategy  $\sigma$ and a prior probability  $\pi$ and furthermore, we focus on subclasses of extensions satisfying some additional analytical properties.
In turn, the properties of disintegrability and strong conglomerability are imposed on the extensions of $\{\pi,\sigma\}$  and are asked to hold with respect to a conditional prior probability extending $\pi$. Hence, the properties of {\it full disintegrability} and {\it full strong conglomerability} are derived, where the latter can be considered as a weakening of the former in case of non-integrability of $\sigma$.

The subclasses of fully disintegrable and fully strongly conglomerable extensions are characterized from a topological point of view, proving they are compact sets with respect to the product topology of pointwise convergence. This implies that the corresponding envelopes are actually attained pointwise by some extension in such subclasses.

The closed form expressions of the lower envelopes of such classes reveal that, on a specific subfamily of conditional events, they are totally monotone capacities. This determines a Choquet integral expression of the envelopes on another distinguished subfamily of conditional events. As a consequence, this also allows to compute as a Choquet integral the corresponding lower (upper) conditional previsions on a suitable class of conditional bounded random variables.

Notice that, since  a prior probability $\pi$ and a {\it statistical model} $\lambda$ give automatically rise to a unique pair $\{\pi,\sigma\}$ where $\sigma$ is a strategy, our results provide a generalized Bayesian updating rule. In particular, it is shown that under the classical hypotheses of Bayes theorem, 
the usual expression of posterior probability is coherent \cite{brr-bayes} but the lower and upper bounds of fully disintegrable extensions could not coincide.


The framework of finitely additive probabilities, as in the present paper, is not forcedly in contrast with countable additivity, as countably additive probabilities constitute a (distinguished) subclass of finitely additive ones. Dealing with finitely additive probabilities is often unavoidable also in the standard Kolmogorovian framework in order to overcome measurability difficulties and, needless to say, to handle  problems admitting a finitely additive solution but not a countably additive one.

For example, as is well-known, the weak limit of a weak convergent net of countably additive probability measures on the Borel $\sigma$-algebra of a topological space is generally only a finitely additive probability. Actually, the finitely additive setting allows to treat successfully some problems arising in weak convergence \cite{br-weakconv}.

Finitely additive probabilities appear ``naturally'' also in statistics \cite{kss-fa-statistics}, for example, to justify the use of the so-called {\it improper priors} \cite{heath1978,heath1989,berti1994} and to connect the profile likelihood with the integrated likelihood \cite{cpv-sma}.

Furthermore, in mathematical finance, it is known that in the {\it fundamental theorem of asset pricing} the absence of arbitrages of the first kind in the market is equivalent to the existence of a finitely additive probability under which the discounted wealth processes become ``local martingales'' \cite{kardaras,bpr-ftap}.

The notion of conditioning adopted in this paper differs from the Kolmogorovian definition as a Radon-Nikodym derivative. In the literature there are several examples showing that the standard notion of conditioning to a sub-$\sigma$-algebra gives rise to counter-intuitive and pathological situations that can be avoided using {\it proper regular conditional distributions} \cite{blackwell1975,ssk-improper-cd} which, in turn, give rise to conditional probabilities in the sense adopted in this paper. Actually, in \cite{blackwell1975} it is shown that countably additive proper regular conditional distributions may not exist for particular $\sigma$-algebras, while existence is proved if only finite additivity is required.

The paper is organized as follows. In Section~\ref{sec:preliminaries} some preliminaries on coherent conditional probabilities, capacities and Choquet integration are provided. In agreement with de~Finetti's approach \cite{definetti2}, we avoid a set-theoretic formalization and work with abstract Boolean algebras of events: since we mainly deal with atomic Boolean algebras, most of the results can be immediately translated in set-theoretic terms. In Section~\ref{sec:coherent-ext} the class of all full conditional probabilities extending $\{\pi,\sigma\}$ is considered and a closed form expression for its envelopes is given. Section~\ref{sec:disintegrable} copes with fully disintegrable extensions, providing a topological characterization of the corresponding class and a closed form expression of its envelopes. Finally, Section~\ref{sec:conglomerable} considers fully strongly conglomerable extensions under the assumption of non-integrability of $\sigma$.

The proofs of results in Section~\ref{sec:coherent-ext}, Section~\ref{sec:disintegrable} and Section~\ref{sec:conglomerable} are collected, respectively, in \ref{appendix0}, \ref{appendix1} and \ref{appendix2}.

\section{Preliminaries}
\label{sec:preliminaries}
Let $\A$ be a Boolean algebra of {\it events}  and denote with $(\cdot)^c$, $\vee$
and $\wedge$ the usual Boolean operations of contrary, disjunction and conjunction, respectively.
The {\it sure event} $\Omega$ and the {\it impossible event} $\emptyset$ coincide, respectively, with the top and bottom elements of $\A$ endowed with the partial order $\subseteq$ of {\it implication}.

A set of events $\G = \{E_i\}_{i \in I}$ can always be embedded into a minimal Boolean algebra denoted as $\langle\G\rangle$ and said the {\it Boolean algebra generated} by $\G$, moreover, $\langle\G\rangle^*$ denotes the {\it complete atomic Boolean algebra generated} by $\G$, which is the minimal atomic Boolean algebra closed under arbitrary conjunctions and disjunctions containing $\G$.  We also denote with $\langle \G \rangle^\sigma$ the {\it Boolean $\sigma$-algebra generated} by $\G$, i.e., the minimal Boolean subalgebra of $\langle\G\rangle^*$ which is closed under countable disjunctions and conjunctions.

A {\it conditional event} $E|H$ is an ordered pair of events $(E,H)$ with $H \neq \emptyset$.
Any set of conditional events $\G = \{E_i|H_i\}_{i \in I}$ can always be embedded into a minimal structured set $\langle\langle\G\rangle\rangle = \A \times \H$, where $\A = \langle\{E_i,H_i\}_{i \in I}\rangle$ and $\H \subseteq \A^0 = \A \setminus \{\emptyset\}$ is an additive class (i.e., it is closed under finite disjunctions).

Recall the definition of {\it coherent conditional probability}  essentially due to de Finetti \cite{definetti,williams,holzer,regazzini}.

\begin{definition}
Let $\G = \{E_i|H_i\}_{i \in I}$ be a set of conditional events. A function $P :\G \to [0,1]$ is a {\bf coherent conditional probability} if and only if, for every $n\in\NN$,
every $E_{i_1}|H_{i_1},\ldots,E_{i_n}|H_{i_n}\in\G$ and every real numbers $s_1,\ldots,s_n$, the random gain
(where ${\bf 1}_E$ denotes the indicator of an event $E$)
\begin{equation}
G=\sum_{j=1}^n s_j({\bf 1}_{E_{i_j}} - P(E_{i_j}|H_{i_j})){\bf 1}_{H_{i_j}},
\end{equation}
satisfies the following inequalities
\begin{equation}
\min_{C_r \subseteq H_0^0} G(C_r)\leq 0\leq  \max_{C_r \subseteq H_0^0} G(C_r),
\end{equation}
where $H_0^0=\bigvee_{j=1}^n H_{i_j}$, and $\B = \langle\{E_{i_j},H_{i_j}\}_{j=1,\ldots,n}\rangle$ with set of atoms $\C_\B = \{C_1,\ldots,C_m\}$.
\end{definition}
In particular, if $\G = \A \times \H$, where $\A$ is a Boolean algebra and $\H \subseteq \A^0$ is an additive class, then $P(\cdot|\cdot)$ is a coherent conditional probability if and only if it satisfies the following conditions:
\begin{description}
\item[{\bf (C1)}] $P(E|H)=P(E\wedge H|H)$, for every $E\in\A$ and $H\in\H$;
\item[{\bf (C2)}] $P(\cdot|H)$ is a finitely additive probability on $\A$, for every $H\in\H$;
\item[{\bf (C3)}] $P(E\wedge F|H)= P(E|H) \cdot P(F|E\wedge H)$, for
every $H, E\wedge H\in \H$ and $E, F\in\A$.
\end{description}
In this case $P(\cdot|\cdot)$ is simply said a {\it conditional probability}, moreover, following the terminology of \cite{dubins}, $P(\cdot|\cdot)$ is said a {\it full conditional probability on $\A$} (or a {\it f.c.p. on $\A$} for short) if its domain is $\A \times \A^0$.


Every coherent conditional probability can be extended on every superset of conditional events by 
the conditional version of the {\it fundamental theorem for probabilities} \cite{definetti,williams,regazzini}.
\begin{theorem}
\label{th:extension}
Let $\G$ and $\G'$ be arbitrary sets of conditional events with $\G \subset \G'$ and $P:\G \to [0,1]$. Then there exists a coherent conditional probability $\tilde{P}(\cdot|\cdot)$ on $\G'$ such that $\tilde{P}_{|\G} = P$ if and only if $P$ is a coherent conditional probability on $\G$. Moreover, if $\G' = \G \cup \{E|H\}$ the coherent values for the conditional probability of $E|H$ range in a closed interval $\II_{E|H} = [\u{P}(E|H),\o{P}(E|H)]$.
\end{theorem}

Theorem~\ref{th:extension} implies that $P:\G \to [0,1]$ is a coherent conditional probability if and only if there exists a conditional probability $\tilde{P}:\langle\langle\G\rangle\rangle \to [0,1]$ extending it.

The interval $\II_{E|H}$ in Theorem~\ref{th:extension} can be explicitly computed in terms of finite subfamilies of $\G$ as
\begin{equation}
\II_{E|H} = \bigcap \left\{\II_{E|H}^\F \,:\, \mbox{$\F \subseteq \G$, $\card \F < \aleph_0$}\right\},
\end{equation}
with $\II_{E|H}^\F = [\u{P}^\F(E|H), \o{P}^\F(E|H)]$ the closed interval obtained extending $P_{|\F}$ on
$\F \cup \{E|H\}$, thus, it holds
\begin{eqnarray}
\u{P}(E|H) &=& \sup\left\{\u{P}^\F(E|H) \,:\, \mbox{$\F \subseteq \G$, $\card \F < \aleph_0$}\right\},\\
\o{P}(E|H) &=& \inf\left\{\o{P}^\F(E|H) \,:\, \mbox{$\F \subseteq \G$, $\card \F < \aleph_0$}\right\}.
\end{eqnarray}

The set $\P = \{\tilde{P}(\cdot|\cdot)\}$ of all the coherent extensions of $P$ to $\G'$ is a compact subset of the space $[0,1]^{\G'}$ endowed with the product topology of pointwise convergence and the projection set on each element of $\G'$ is a (possibly degenerate) closed interval. The pointwise envelopes
\begin{equation}
\u{P}=\min\P\quad\mbox{and}\quad\o{P}=\max\P,
\end{equation}
are known as {\it coherent lower} and {\it upper conditional probabilities}, respectively \cite{williams}. The envelopes $\u{P}$ and $\o{P}$ satisfy the {\it duality} property, i.e., $\o{P}(E|H) = 1 - \u{P}(E^c|H)$, for every $E|H,E^c|H \in \G'$.

A {\it (normalized) capacity} on a Boolean algebra $\A$ is a function $\varphi:\A \to [0,1]$ satisfying $\varphi(\emptyset) = 0$, $\varphi(\Omega) = 1$ and $\varphi(A) \le \varphi(B)$ when $A \subseteq B$, for $A,B \in \A$. Every capacity has an associated {\it dual} capacity $\psi$ defined for every $E \in \A$ as $\psi(E) = 1 - \varphi(E^c)$. A capacity $\varphi$ is then said to be {\it $n$-monotone}, with $n \ge 2$, if it satisfies the further property, for every $A_1,\ldots,A_n \in \A$,
\begin{equation}
\varphi\left(\bigvee_{i=1}^n A_i\right) \ge \sum\limits_{\emptyset \neq I \subseteq \{1,\ldots,n\}}(-1)^{|I| - 1}
\varphi\left(\bigwedge_{i \in I}A_i\right).
\label{eq:n-mon}
\end{equation}
A capacity which is $n$-monotone for every $n \ge 2$ is also referred to as {\it totally monotone capacity}.

Every $2$-monotone capacity $\varphi$ on $\A$ induces a closed convex set of finitely additive probabilities on $\A$ said {\it core} defined as
\begin{equation}
\P_\varphi = \{\tilde{\pi} \,:\, \mbox{$\tilde{\pi}$ finitely additive probability on $\A$, $\varphi \le \tilde{\pi}$}\}.
\end{equation}

Every capacity $\varphi$ on $\A$ gives rise to an {\it inner measure} $\varphi_*$ on $\langle\A\rangle^*$ defined for every $E \in \langle\A\rangle^*$ as
$
\varphi_*(E) = \sup\{\varphi(B) \,:\, B \subseteq E, B \in \A\}.
$
In \cite{choquet,denneberg-inner} it is proved that if $\varphi$ is $n$-monotone then also $\varphi_*$ is, so the inner measure induced by a finitely additive probability is always totally monotone. 
The dual of the inner measure induced by a capacity will be referred to as {\it outer measure} and will be denoted as $\varphi^*$.

Let $\L=\{H_i\}_{i \in I}$ be a {\it partition} of $\Omega$,
$\A_\L$ a Boolean algebra such that $\langle \L \rangle \subseteq \A_\L \subseteq \langle \L \rangle^*$, and $X:\L \to \RR$ a real function. For $t \in \RR$, introduce the event $(X \ge t) =\bigvee\{H_i \in \L\;:\; X(H_i) \ge t\}$, which does not necessarily belong to $\A_\L$.

Given a finitely additive probability $\pi$ on $\A_\L$, the {\it lower} and {\it upper Stieltjes integrals} with respect to $\pi$ (see, e.g., \cite{bhaskara}) of a bounded $X$ are defined as
\begin{eqnarray}
\u{\int}X(H_i)\pi(\d H_i) &=& \sup_{\L^\F \subseteq \A_\L} \sum_{u=1}^m \left(\left(\inf_{H_i \subseteq A_u} X(H_i)\right)\pi(A_u)\right),\\
\o{\int}X(H_i)\pi(\d H_i) &=& \inf_{\L^\F \subseteq \A_\L} \sum_{u=1}^m \left(\left(\sup_{H_i \subseteq A_u} X(H_i)\right)\pi(A_u)\right),
\end{eqnarray}
where $\L^\F = \{A_u\}_{u=1}^m$ is a finite partition of $\Omega$ contained in $\A_\L$.
The function $X$ is said {\it Stieltjes integrable}, or {\it S-integrable} for short, with respect to $\pi$ if $\u{\int}X(H_i)\pi(\d H_i) = \o{\int}X(H_i)\pi(\d H_i)$ and their common value, denoted as $\int X(H_i)\pi(\d H_i)$, is called {\it Stieltjes integral}.

The function $X$ is said {\it $\A_\L$-continuous} \cite{bhaskara,deCooman-book} if is bounded and for every $t \in \RR$ and $\epsilon > 0$ there exists $A \in \A_\L$ such that
$
(X \ge t) \supseteq A \supseteq (X \ge t + \epsilon).
$
The notion of $\A_\L$-continuity coincides with the notion of measurability required in \cite{schmeidler}, moreover, every $\A_\L$-continuous function is S-integrable with respect to every finitely additive probability on $\A_\L$ \cite{bhaskara}. If $\A_\L$ is a Boolean $\sigma$-algebra then $\A_\L$-continuous functions exactly coincide with bounded $\A_\L$-measurable (in the usual sense) functions.

Given a capacity $\varphi$ on $\A_\L$ with associated inner measure $\varphi_*$ on $\langle\L\rangle^*$, the {\it Choquet integral} of an $\A_\L$-continuous $X$ (see, e.g., \cite{denneberg}) with respect to $\varphi$ is defined as
\begin{equation}
\intC X(H_i) \varphi(\d H_i) = \int_{-\infty}^0 [\varphi_*(X \ge t) - 1] \d t + \int_0^{+\infty} \varphi_*(X \ge t) \d t,
\end{equation}
where the integrals on the right side are usual Riemann integrals.

For every $2$-monotone capacity $\varphi$ on $\A$ with dual capacity $\psi$ and core $\P_\varphi$, and every $\A_\L$-continuous function $X$ it holds (see, e.g., \cite{schmeidler})
\begin{eqnarray}
\intC X(H_i) \varphi(\d H_i) &=& \min\left\{\int X(H_i)\tilde{\pi}(\d H_i) \,:\, \tilde{\pi} \in \P_\varphi\right\},\\
\intC X(H_i) \psi(\d H_i) &=& \max\left\{\int X(H_i)\tilde{\pi}(\d H_i) \,:\, \tilde{\pi} \in \P_\varphi\right\},
\end{eqnarray}
so, in particular, if $\varphi$ is finitely additive it holds $\intCs X \d \varphi = \intCs X \d \psi = \int X \d \varphi$.

\section{Coherent extensions of a strategy and a prior}
\label{sec:coherent-ext}
Let $\L = \{H_i\}_{i \in I}$ and $\E = \{E_j\}_{j \in J}$ be two partitions of $\Omega$, with $I,J$ arbitrary index sets, and consider the Boolean algebras $\A_\L$ and $\A_\E$ such that $\langle\L\rangle \subseteq \A_\L \subseteq \langle\L\rangle^*$ and  $\langle\E\rangle \subseteq \A_\E \subseteq \langle\E\rangle^*$.

Consider a Boolean algebra $\A$ such that $\langle\A_\L \cup \A_\E\rangle \subseteq \A \subseteq \langle\A_\L \cup \A_\E\rangle^*$, whose set of atoms is $\C_\A = \{H_i \wedge E_j \neq \emptyset \,:\, H_i \in \L, E_j \in \E\}$.

A {\it strategy} is any map $\sigma:\A \times \L \to [0,1]$ satisfying the following conditions for every $H_i \in \L$:
\begin{itemize}
\item[\bf (S1)]
$\sigma(F|H_i) = 1$ if $F \wedge H_i = H_i$, for every $F \in \A$;
\item[\bf (S2)] $\sigma(\cdot|H_i)$ is a finitely additive probability on $\A$.
\end{itemize}

It is known (see \cite{dubins,regazzini-bayes}) that any strategy is a coherent conditional probability and the same holds for any {\it statistical model} $\lambda = \sigma_{|\A_\E \times \L}$. In general, given a statistical model $\lambda$ on $\A_\E \times \L$ there can exist possibly infinite strategies extending $\lambda$ on $\A \times \L$. Nevertheless, in the case $\A$ coincides with the Boolean algebra generated by $\A_\L \cup \A_\E$, a statistical model $\lambda$ extends uniquely to a strategy $\sigma$ on $\A \times \L$.

\begin{proposition}
\label{prop:uniqueness-strategy}
Let $\lambda$ be a statistical model on $\A_\E \times \L$ and $\A = \langle\A_\L \cup \A_\E\rangle$, then there exists a unique strategy $\sigma$ on $\A \times \L$ such that $\sigma_{|\A_\E \times \L} = \lambda$.
\end{proposition}

For the sake of generality, in what follows we will always refer to a strategy $\sigma$ in place of a statistical model $\lambda$.

As immediate consequence of Theorem~5 in \cite{dubins}, given a strategy $\sigma$ on $\A \times \L$ and a finitely additive prior probability $\pi$ on $\A_\L$, the whole assessment $\{\pi,\sigma\}$ is a coherent conditional probability on $\G = (\A_\L \times \{\Omega\}) \cup (\A \times \L)$.

A finitely additive probability $P^\j:\A \to [0,1]$ is said a {\it joint probability consistent with $\{\pi,\sigma\}$} if
$P^\j_{|\A_\L} = \pi$ and $\{P^\j,\sigma\}$ is a coherent conditional probability on $\G' = \A \times (\{\Omega\} \cup \L)$. In general the joint probability on $\A$ consistent with $\{\pi,\sigma\}$ is not unique, so the aim here is to characterize the set of consistent joint probabilities $\P^\j = \{\tilde{P}^\j(\cdot)\}$ in terms of its envelopes that are called {\it lower} and {\it upper joint probabilities} $\u{P}^\j = \min \P^\j$ and $\o{P}^\j = \max \P^\j$.

For this, consider the set
$$
\P = \{\tilde{P} \,:\, \mbox{f.c.p. on $\A$ extending $\{\pi,\sigma\}$}\},
$$
together with its lower and upper envelopes $\u{P} = \min \P$ and $\o{P} = \max \P$, for which it holds $\u{P}^\j = \u{P}_{|\A \times \{\Omega\}}$ and $\o{P}^\j = \o{P}_{|\A \times \{\Omega\}}$.

The following theorem provides a characterization of the lower joint probability $\u{P}^\j$ on $\A$ consistent with $\{\pi,\sigma\}$.

\begin{theorem}
\label{th:joint}
The lower joint probability $\u{P}^\j(\cdot)$ is such that, for every $F \in \A$, it holds
$$
\u{P}^\j(F) = \sup_{\L^\F \subseteq \A_\L}\left\{\sum_{h=1}^n \sigma(F|H_{i_h})\pi(H_{i_h}) + \sum_{B_k \subseteq F} \pi(B_k)\right\},
$$
where $\L^\F = \{H_{i_h}\}_{h=1}^n \cup \{B_k\}_{k = 1}^t$ is a finite partition of $\Omega$ contained in $\A_\L$.
\end{theorem}

If $\L$ is countable and $\pi$ is countably additive on $\A_\L$, then for every $F \in \A$ it holds
$
\u{P}^\j(F) = \sum_{i=1}^{\infty}\sigma(F|H_i)\pi(H_i),
$
thus, $\u{P}^\j$ is a finitely additive probability on $\A$ coinciding with $\o{P}^\j$ (i.e., $\P^\j$ reduces to a singleton), moreover, if $\sigma(\cdot|H_i)$ is countably additive on $\A$ for every $H_i \in \L$, then $\u{P}^\j$ is countably additive.
On the contrary, if $\card \L > \aleph_0$, then the countable additivity of $\pi$ does not imply the unicity of the joint probability in $\P^\j$ as showed by the following example.

\begin{example}
Let $\L = \{H_i\}_{i \in \RR}$, $\E = \{E_1,E_2\}$ with $H_i \wedge E_j \neq \emptyset$ for every $i,j$, and take $\A_\L$ isomorphic to the Borel $\sigma$-field on $\RR$, $\A_\E = \langle\E\rangle$, $\A = \langle \A_\L \cup \A_\E\rangle$. Let $\pi$ be any diffuse (i.e., such that $\pi(H_i) = 0$ for every $i$) countable additive probability on $\A_\L$ and $\sigma$ any strategy on $\A \times \L$. For every finite partition $\L^\F = \{H_{i_h}\}_{h=1}^n \cup \{B_k\}_{k = 1}^t$ it holds $\sigma(E_j|H_{i_h})\pi(H_{i_h}) = 0$ and $\emptyset \neq E_j \wedge B_k \neq E_j$, for $j=1,2$, thus $\u{P}^\j(E_j) = 0$ and $\o{P}^\j(E_j) = 1$ for $j=1,2$.
\end{example}

Next result provides a characterization of $\u{P}(\cdot|\cdot)$ relying on $\u{P}^\j(\cdot)$, $\o{P}^\j(\cdot)$ and the functions $L^\j(\cdot,\cdot)$ and
$U^\j(\cdot,\cdot)$ defined for $F \in \A$ and $K \in \A^0$ as
\begin{eqnarray}
L^\j(F,K) &=& \min\left\{\tilde{P}^\j(F \wedge K) \,:\, \tilde{P}^\j \in \P^\j,\right.\\\nonumber
&&~~~~~~~~~~~~\left.\tilde{P}^\j(F^c \wedge K) = \o{P}^\j(F^c \wedge K)\right\},\\
U^\j(F,K) &=& \max\left\{\tilde{P}^\j(F \wedge K) \,:\, \tilde{P}^\j \in \P^\j,\right.\\\nonumber
&&~~~~~~~~~~~~\left.\tilde{P}^\j(F^c \wedge K) = \u{P}^\j(F^c \wedge K)\right\},
\end{eqnarray}
for which it holds
$$
\u{P}^\j(F \wedge K) \le L^\j(F,K) \le U^\j(F,K) \le \o{P}^\j(F \wedge K).
$$

\begin{theorem}
\label{th:ext-prob-lik-inf}
The lower envelope $\u{P}(\cdot|\cdot)$ is such that, for every $F|K \in \A \times \A^0$, $\u{P}(F|K) = 1$ when $F \wedge K = K$, and if $F \wedge K \neq K$, then:
\begin{itemize}
\item[\it (i)] if $\u{P}^\j(K) > 0$, then
$$
\u{P}(F|K) =\min\left\{\frac{\u{P}^\j(F \wedge K)}{\u{P}^\j(F \wedge K) + U^\j(F^c,K)},
\frac{L^\j(F,K)}{L^\j(F,K) + \o{P}^\j(F^c \wedge K)}\right\};
$$
\item[\it (ii)] if $\u{P}^\j(K) = 0$, then
$$
\u{P}(F|K) =
\left\{
\begin{array}{ll}
\min\limits_{i \in I_2^{F|K}} \frac{\sigma(F \wedge K|H_i)}{\sigma(K|H_i)} & \mbox{if $I_2^{F|K} \neq \emptyset = I_3^{F|K}$, $\card I_2^{F|K} < \aleph_0$}\\
& \mbox{and $\sigma(K|H_i) > 0$ for all $i \in I_2^{F|K}$},\\[2ex]
0 & \mbox{otherwise},
\end{array}
\right.
$$
where
$I_1^{F|K} = \{i \in I \,:\, H_i \wedge F \wedge K \neq \emptyset =  H_i \wedge F^c \wedge K\}$,
$I_2^{F|K} = \{i \in I \,:\, H_i \wedge F \wedge K \neq \emptyset \neq  H_i \wedge F^c \wedge K\}$, and
$I_3^{F|K} = \{i \in I \,:\, H_i \wedge F \wedge K = \emptyset \neq  H_i \wedge F^c \wedge K\}$.
\end{itemize}
\end{theorem}

The following example, inspired to Example~2.1 in \cite{regazzini-unimi}, shows the application of previous theorem in a Bayesian inferential procedure.

\begin{example}
\label{ex:coerenti}
Consider a finite population of unknown size and let $\Theta$ be the random relative frequency of a characteristic under study. For the range of $\Theta$ it is ``natural'' to assume ${\bf \Theta} = [0,1] \cap \QQ$, so let $\L = \{H_{\theta} = (\Theta = \theta) \,:\, \theta \in {\bf \Theta}\}$, and $\A_\L = \langle\L\rangle^*$.

Assign a uniform distribution to $\Theta$, specifying a prior probability $\pi$
on $\A_\L$ such that $\pi(\Theta \in [0,\theta] \cap {\bf \Theta}) = \theta$, for $\theta \in {\bf \Theta}$.
The probability $\pi$ is only finitely additive since, for $\theta \in {\bf \Theta}$,
$(\Theta = \theta) = \bigwedge_{n \in \NN} \left(\Theta \in \left(\theta-\frac{1}{n},\theta\right] \cap {\bf \Theta}\right)$, thus
$\pi(\Theta = \theta) = \lim_{n \in \NN} \pi\left(\Theta \in \left(\theta-\frac{1}{n},\theta\right]  \cap {\bf \Theta}\right) = 0$.

Draw a sample with replacement of size $n$ from the population, and let $X$ be the number of individuals showing the characteristic under study, whose range is ${\bf X} = \{0,\ldots,n\}$. Let $\E = \{E_x = (X = x) \,: \, x \in {\bf X}\}$ for which it holds $H_0 \subseteq E_0$, $H_1 \subseteq E_n$, and $H_\theta \wedge E_x \neq \emptyset$ for $\theta \in {\bf \Theta} \setminus \{0,1\}$ and $x \in {\bf X}$. Take $\A_\E = \langle\E\rangle$ and let $\lambda$ be the statistical model on $\A_\E \times \L$ singled out by
$$
\lambda(X = x|\Theta = \theta) = \binom{n}{x}\theta^x (1 - \theta)^{n-x}.
$$
The statistical model $\lambda$ uniquely extends to a strategy $\sigma$ on $\A \times \L$, where $\A = \langle\A_\L \cup \A_\E\rangle$.

Consider the conditional event $A|B$ with $A = \left(\Theta \in \left(\theta_1,\theta_2\right] \cap {\bf \Theta} \right)$ and $B = (X = n)$, where $\theta_1,\theta_2 \in {\bf \Theta}$, $\theta_1 < \theta_2$, which is related to a posterior probability. Since $H_0 \wedge E_n = \emptyset$, $H_1 \subseteq E_n$, $H_\theta \wedge E_n \neq \emptyset$ for $\theta \in {\bf \Theta} \setminus \{0,1\}$, and $\pi(\Theta = \theta) = 0$ for every $\theta \in {\bf \Theta}$, we have $\u{P}^\j(B) = 0$. Moreover, since $I_3^{A|B} \neq \emptyset \neq I_3^{A^c|B}$, it follows $\u{P}(A|B) = \u{P}(A^c|B) = 0$, and so $\o{P}(A|B) = 1$, which implies that the coherent probability values of $A|B$ range in $[0,1]$.

Then, consider the conditional event $C|D$ with $C = (X = n)$ and $D = \left(\Theta \in \left\{\frac{i}{10} \,:\, i=1,\ldots,9\right\}\right)$.
We have $\u{P}^\j(D) = 0$, $I_1^{C|D} = I_3^{C|D} =
I_1^{C^c|D} = I_3^{C^c|D} = \emptyset$, and $I_2^{C|D} = I_2^{C^c|D} = \left\{\frac{i}{10} \,:\, i=1,\ldots,9\right\}$. Since $\sigma\left(D|\Theta = \frac{i}{10}\right) = 1$, for $i=1,\ldots,9$, it follows $\u{P}(C|D) = \left(\frac{1}{10}\right)^n$, $\u{P}(C^c|D)= 1 - \left(\frac{9}{10}\right)^n$, and so $\o{P}(C|D) = \left(\frac{9}{10}\right)^n$, which implies that the coherent probability values of $C|D$ range in $\left[\left(\frac{1}{10}\right)^n,\left(\frac{9}{10}\right)^n\right]$.
\end{example}

\section{Fully disintegrable extensions}
\label{sec:disintegrable}
Given a strategy $\sigma$ on $\A \times \L$ and a finitely additive prior probability $\pi$ on $\A_\L$ defined as in Section~\ref{sec:coherent-ext}, it has already been stressed that the joint probability on $\A$ consistent with $\{\pi,\sigma\}$ is generally not unique.

When the function $\sigma(F|\cdot)$ is S-integrable with respect to $\pi$ for every $F \in \A$, there is \cite{dubins,regazzini-bayes,brr-bayes} a joint probability $P^\dd$ in $\P^\j$, defined for every $F\in\A$ as
\begin{equation}
P^\dd(F)=\int \sigma(F|H_i)\pi(\d H_i).
\label{eq:dis}
\end{equation}

To avoid cumbersome measurability requirements, in the rest of this section assume that, for every $F \in \A$, $\sigma(F|\cdot)$ is $\A_\L$-continuous. In particular, if $\A_\L = \langle\L\rangle^*$ then no measurability requirement has to be imposed on $\sigma(F|\cdot)$.

The main feature of $P^\dd$ is that it is {\it $\L$-disintegrable} \cite{dubins} with respect to the prior $\pi$ and the strategy $\sigma$, however, as claimed in \cite{brr-bayes}, this is just one of the possible joint probabilities on $\A$ consistent with $\{\pi,\sigma\}$. We refer to this particular element of $\P^\j$ as {\it $\L$-disintegrable joint probability}.

Since the assessment $\{P^\dd,\sigma\}$ is coherent, it can be extended further to $\A \times \A^0$.
The extension is uniquely determined for all the events $F|K \in \A \times \A^0$ such that $P^\dd(K) > 0$, while uniqueness is lost in case $P^\dd(K) = 0$, for this introduce the set
$$
\P^\dd = \{\tilde{P} \,:\, \mbox{f.c.p. on $\A$ extending $\{P^\dd,\sigma\}$}\},
$$
whose lower envelope is $\u{P}^\dd = \min\P^\dd$.

\begin{theorem}
\label{th:prior-lik-inf-dis}
The lower envelope $\u{P}^\dd(\cdot|\cdot)$ is such that, for every $F|K \in \A \times \A^0$, $\u{P}^\dd(F|K) = 1$ when $F \wedge K = K$, and if $F \wedge K \neq K$, then:
\begin{itemize}
\item[\it (i)] if $P^\dd(K) > 0$, then
$$
\u{P}^\dd(F|K) =\frac{P^\dd(F \wedge K)}{P^\dd(K)};
$$
\item[\it (ii)] if $P^\dd(K) = 0$, then
$$
\u{P}^\dd(F|K) =
\left\{
\begin{array}{ll}
\min\limits_{i \in I_2^{F|K}} \frac{\sigma(F \wedge K|H_i)}{\sigma(K|H_i)} & \mbox{if $I_2^{F|K} \neq \emptyset = I_3^{F|K}$, $\card I_2^{F|K} < \aleph_0$}\\
& \mbox{and $\sigma(K|H_i) > 0$ for all $i \in I_2^{F|K}$},\\[2ex]
0 & \mbox{otherwise},
\end{array}
\right.
$$
where
$I_1^{F|K} = \{i \in I \,:\, H_i \wedge F \wedge K \neq \emptyset =  H_i \wedge F^c \wedge K\}$,
$I_2^{F|K} = \{i \in I \,:\, H_i \wedge F \wedge K \neq \emptyset \neq  H_i \wedge F^c \wedge K\}$, and
$I_3^{F|K} = \{i \in I \,:\, H_i \wedge F \wedge K = \emptyset \neq  H_i \wedge F^c \wedge K\}$.
\end{itemize}
\end{theorem}

Even restricting to the $\L$-disintegrable joint probability $P^\dd$, the set of extensions of $\{P^\dd,\sigma\}$ on $\A \times \A^0$ can give rise to non-informative probability bounds (i.e., reducing to $0$ and $1$, respectively) for a large class of conditional events.

For this, the aim now is to restrict further the set of coherent extensions of $\{P^\dd,\sigma\}$ by selecting only those extensions on $\A \times \A^0$ satisfying the following stronger notion of $\L$-disintegrability.

\begin{definition}
A full conditional probability $\tilde{Q}(\cdot|\cdot)$ on $\A$ extending $\{\pi,\sigma\}$ is {\bf fully $\L$-disintegrable} if, denoting with $\tilde{\pi}^\p = \tilde{Q}_{|\A_\L \times \A_\L^0}$, for every $F|K \in \A \times \A_\L^0$ it holds
\begin{equation}
\tilde{Q}(F|K) = \int \sigma(F|H_i)\tilde{\pi}^\p(\d H_i|K).
\end{equation}
\end{definition}

Hence, let us consider the set
$$
\Q^\fd = \{\tilde{Q} \,:\, \mbox{fully $\L$-disintegrable f.c.p. on $\A$ extending $\{\pi,\sigma\}$}\},
$$
whose topological structure is investigated in the following theorem.

\begin{theorem}
\label{th:compact}
The set $\Q^\fd$ is a non-empty compact subset of the space $[0,1]^{\A \times \A^0}$ endowed with the product topology of pointwise convergence.
\end{theorem}

The following result characterizes the lower envelope $\u{Q}^\fd = \min \Q^\fd$, relying on the following functions, defined for $F \in \A$, $K \in \A^0$ and $A \in \A_\L^0$ with $K \subseteq A$ as
\begin{eqnarray}
L^\fd(F, K;A) &=& \min\left\{\tilde{Q}(F \wedge K|A) \,:\, \tilde{Q} \in \Q^\fd,\right.\label{eq:Ld-full}\\\nonumber
&&~~~~~~~~~~~~\left.\tilde{Q}(F^c \wedge K|A) = \o{Q}^\fd(F^c \wedge K|A)\right\},\\
U^\fd(F, K;A) &=& \max\left\{\tilde{Q}(F \wedge K|A) \,:\, \tilde{Q} \in \Q^\fd,\right.\label{eq:Ud-full}\\\nonumber
&&~~~~~~~~~~~~\left.\tilde{Q}(F^c \wedge K|A) = \u{Q}^\fd(F^c \wedge K|A)\right\},
\end{eqnarray}
that can be equivalently expressed as in equations (\ref{eq:Ld}) and (\ref{eq:Ud}) in \ref{appendix1}, and
for which it holds
$$
\u{Q}^\fd(F \wedge K|A) \le L^\fd(F, K;A) \le U^\fd(F, K;A) \le \o{Q}^\fd(F \wedge K|A).
$$

\begin{theorem}
\label{th:fully-L-dis-fcp}
The lower envelope $\u{Q}^\fd(\cdot|\cdot)$ is such that for every $F|K \in \A \times \A^0$, $\u{Q}^\fd(F|K) = 1$ when $F \wedge K = K$, and if $F \wedge K \neq K$, then
\begin{itemize}
\item[\it (i)] if $K \in \A_\L^0$
$$
\u{Q}^\fd(F|K) =
\left\{
\begin{array}{ll}
\displaystyle{\frac{\int \sigma(F \wedge K|H_i) \pi(\d H_i)}{\pi(K)}} & \mbox{if $\pi(K) > 0$},\\[4ex]
\inf\limits_{H_i \subseteq K} \sigma(F|H_i) & \mbox{otherwise},
\end{array}
\right.
$$
\item[\it (ii)] if $K \in \A^0 \setminus \A_\L^0$, then if there exists $A \in \A_\L^0$ such that $K \subseteq A$ and $\u{Q}^\fd(K|A) > 0$ we have that
\begin{eqnarray*}
\u{Q}^\fd(F|K) &=& \min\left\{\frac{\u{Q}^\fd(F \wedge K|A)}{\u{Q}^\fd(F \wedge K|A) + U^\fd(F^c,K;A)},\right.\\
&&~~~~~~~~~~~~\left.\frac{L^\fd(F, K;A)}{L^\fd(F, K;A) + \o{Q}^\fd(F^c \wedge K|A)}\right\},
\end{eqnarray*}
otherwise $\u{Q}^\fd(F|K) = 0$.
\end{itemize}
\end{theorem}

Note that, in condition {\it (i)} of Theorem~\ref{th:fully-L-dis-fcp}, $\u{Q}^\fd(\cdot|\cdot)$ can be expressed as a suitable Choquet integral (see Lemma~\ref{lem:lower-fully-L-dis} in \ref{appendix1}).

Next example reviews Example~\ref{ex:coerenti} focusing only on the fully $\L$-disintegrable extensions.

\begin{example}
\label{ex:fully-L-dis}
Consider the situation described in Example~\ref{ex:coerenti}. Since $\A_\L = \langle\L\rangle^*$ it immediately follows that for every $F \in \A$, $\sigma(F|\cdot)$ is an $\A_\L$-continuous function on $\L$.

For the conditional event $A|B$ (see also \cite{regazzini-unimi}), since $B \notin \A_\L^0$, it holds
\begin{eqnarray*}
\u{Q}^\fd(B|\Omega) &=& \o{Q}^\fd(B|\Omega) = \int \sigma(B|H_\theta)\pi(\d H_\theta) = \frac{1}{n+1},\\
\u{Q}^\fd(A \wedge B|\Omega) &=& \o{Q}^\fd(A \wedge B|\Omega) = \int \sigma(A \wedge B|H_\theta)\pi(\d H_\theta) =
\frac{\theta_2^{n+1} - \theta_1^{n+1}}{n+1}.
\end{eqnarray*}
Thus we have
$
\u{Q}^\fd(A \wedge B|\Omega) + U^\fd(A^c,B;\Omega) = L^\fd(A, B;\Omega) + \o{Q}^\fd(A^c \wedge B|\Omega) = \frac{1}{n+1},
$
which implies $\u{Q}^\fd(A|B) = \o{Q}^\fd(A|B) = \theta_2^{n+1} - \theta_1^{n+1}$.

For the conditional event $C|D$, since $D \in \A_\L^0$ and $\pi(D) = 0$, it follows $\u{Q}^\fd(C|D) = \left(\frac{1}{10}\right)^n$, $\u{Q}^\fd(C^c|D)= 1 - \left(\frac{9}{10}\right)^n$, and so $\o{Q}^\fd(C|D) = \left(\frac{9}{10}\right)^n$ that coincide with the probability bounds determined by the whole set of coherent extensions.
\end{example}

Recall that Theorem~\ref{th:compact} implies $\u{Q}^\fd$ is attained pointwise by at least an extension in $\Q^\fd$ so, in particular, there is at least an extension in $\Q^\fd$ assuming the infimum in condition {\it (i)} of Theorem~\ref{th:fully-L-dis-fcp} as shown by next example.
\begin{example}
Consider the partitions $\L = \{H_i\}_{i \in \NN}$ and $\E = \{E_1,E_2\}$ with $H_i \wedge E_j \neq \emptyset$ for every $i,j$. Take $\A_\L = \langle\L\rangle^*$, $\A_\E = \langle\E\rangle$, $\A = \langle\A_\L \cup \A_\E\rangle$ and let $\U$ be an ultrafilter of $\A_\L$ containing $B = \bigvee_{i \in \NN} H_{2i-1}$. 

Let $\pi$ the finitely additive prior probability defined for $K \in \A_\L$ as $\pi(K)=1$ if $K \in \U$ and $0$ otherwise, and $\lambda$ the statistical model on $\A_\E \times \L$ singled out for $i \in \NN$ by
$\lambda(E_1|H_i) = \frac{1}{2} + \frac{1}{2i}$ and $\lambda(E_2|H_i) = 1 - \lambda(E_1|H_i)$,
which extends uniquely to a strategy $\sigma$ on $\A \times \L$. 

Since $\pi(B^c) = 0$ it holds
\begin{eqnarray*}
\u{Q}^\fd(E_1|B^c) &=& \inf\limits_{H_i \subseteq B^c} \sigma(E_1|H_i) = \frac{1}{2},\\
\o{Q}^\fd(E_1|B^c) &=& 1 - \u{Q}^\fd(E_2|B^c) = 1 - \inf\limits_{H_i \subseteq B^c} \sigma(E_2|H_i) = \frac{3}{4}.
\end{eqnarray*}
So, there exist $\tilde{Q}_1, \tilde{Q}_2 \in \Q^\fd$ such that $\tilde{Q}_1(E_1|B^c) = \frac{1}{2}$ and $\tilde{Q}_2(E_1|B^c) = \frac{3}{4}$.
\end{example}

Let us consider the case where $H_i \wedge E_j \neq \emptyset$ for every $i,j$, $\A_\L$ and $\A_\E$ are Boolean $\sigma$-algebras, and $\A = \langle\A_\L \cup \A_\E\rangle^\sigma$. If $\pi$ is countably additive on $\A_\L$ and $\lambda$ is a statistical model on $\A_\E \times \L$ such that $\lambda(\cdot|H_i)$ is countably additive and absolutely continuous with respect to the same $\sigma$-finite measure $\mu$ on $\A_\E$ for every $H_i \in \L$, then $l(\cdot;H_i)$ is the Radon-Nikodym derivative of $\lambda(\cdot|H_i)$ with respect to $\mu$. Under previous hypotheses the function $\lambda$ is also called a {\it transition kernel} and $l(E_j;\cdot)$ is assumed to be $\A_\L$-measurable for every $E_j \in \E$.

In general, Proposition~\ref{prop:uniqueness-strategy} only implies that $\lambda$ uniquely extends to a strategy on
$\langle\A_\L \cup \A_\E\rangle \times \L$. Nevertheless, since $\A_\L$ and $\A_\E$ are Boolean $\sigma$-algebras and $\A$ is the Boolean $\sigma$-algebra generated by them, fixing $H_i \in \L$, for every $F \in \A$ there exists $F_{H_i} \in \A_\E$ such that $F \wedge H_i = F_{H_i}  \wedge H_i$. Thus for every strategy $\sigma$ on $\A \times \L$ extending $\lambda$ it must hold
$$
\sigma(F|H_i) = \sigma(F \wedge H_i|H_i) = \sigma(F_{H_i} \wedge H_i|H_i) = \sigma(F_{H_i} |H_i) = \lambda(F_{H_i} |H_i),
$$
so $\sigma$ is uniquely determined by $\lambda$, $\sigma(\cdot|H_i)$ is countably additive on $\A$ for every $H_i \in \L$,
and $\sigma(F|\cdot)$ is bounded and $\A_\L$-measurable for every $F \in \A$. In turn, this implies
$\u{Q}^\fd(\cdot|\Omega) = \o{Q}^\fd(\cdot|\Omega)$ is countably additive on $\A$, moreover,
for every $B|E_j \in \A_\L \times \E$  such that $0 < \int l(E_j;H_i) \pi(\d H_i) < +\infty$ it holds
$$
\u{Q}^\fd(B|E_j) \le \frac{\int l(E_j;H_i) {\bf 1}_{B}(H_i) \pi(\d H_i)}{\int l(E_j;H_i) \pi(\d H_i)}
\le \o{Q}^\fd(B|E_j),
$$
where the involved integrals are in the Lebesgue sense.
If further $\pi$ is absolutely continuous with respect to a $\sigma$-finite measure $\nu$ on $\A_\L$, then $p$ is the Radon-Nikodym derivative of $\pi$ with respect to $\nu$ and we obtain 
that the usual statement of Bayes theorem for densities produces a coherent value \cite{brr-bayes} even if the inequalities above may be strict.

\begin{example}
Consider two random variables $\Theta$ and $X$ ranging, respectively, on ${\bf \Theta} = {\bf X} = [0,1]$, and let $\L = \{H_\theta = (\Theta = \theta) \;:\, \theta \in {\bf \Theta}\}$ and 
$\E = \{E_x = (X = x) \;:\, x \in {\bf X}\}$ with $H_\theta \wedge E_x \neq \emptyset$ for every $\theta,x$. Let $\A_\L$ and $\A_\E$ be isomorphic to the Borel $\sigma$-fields on ${\bf \Theta}$ and ${\bf X}$, respectively, and $\A = \langle\A_\L \cup \A_\E\rangle^\sigma$.

Let $\pi$ and $\lambda(\cdot|H_\theta)$, for every $\theta \in {\bf \Theta}$, coincide with the Lebesgue measure on $[0,1]$. The statistical model extends uniquely to a strategy $\sigma$ on $\A \times \L$.

We want to compute $\u{Q}^\fd(B|E_{0.5})$ where $B = (\Theta \in [0,0.5])$ and $E_{0.5} = (X = 0.5)$. In order to be $\u{Q}^\fd(B|E_{0.5}) \neq 0$ there must exist $A \in \A_\L^0$ such that $\u{Q}^\fd(E_{0.5}|A) > 0$ where
$$
\u{Q}^\fd(E_{0.5}|A) = \left\{
\begin{array}{ll}
\frac{\int \sigma(E_{0.5} \wedge A|H_\theta) \pi(\d H_\theta)}{\pi(A)} = \frac{\int \lambda(E_{0.5}|H_\theta) {\bf 1}_A(H_\theta) \pi(\d H_\theta)}{\pi(A)} & \mbox{if $\pi(A) > 0$},\\[2ex]
\inf\limits_{H_\theta \subseteq A} \lambda(E_{0.5}|H_\theta) & \mbox{otherwise}.
\end{array}
\right.
$$

Notice that $l(E_x ; H_\theta) = 1$ for every $x,\theta$ and $\int l(E_x ; H_\theta) \pi(\d H_\theta) = 1 \in (0,+\infty)$.
Since $\lambda(E_{0.5}|H_\theta) = \int_{0.5}^{0.5} \d x = 0$, for every $\theta \in {\bf \Theta}$, it trivially holds that $\u{Q}^\fd(E_{0.5}|A) = 0$ for every $A \in \A_\L^0$, so it must be $\u{Q}^\fd(B|E_{0.5}) = 0$. Similarly, it is possible to show that $\o{Q}^\fd(B|E_{0.5}) = 1 - \u{Q}^\fd(B^c|E_{0.5}) = 1$.
\end{example}


\section{Fully strongly conglomerable extensions}
\label{sec:conglomerable}
Consider a strategy $\sigma$ on $\A \times \L$ and a finitely additive prior probability $\pi$ on $\A_\L$ as in Section~\ref{sec:coherent-ext}.

Now we focus on strongly $\L$-conglomerable joint probabilities consistent with $\{\pi,\sigma\}$, according to the following definition \cite{brr-bayes}.

\begin{definition}
\label{def:cong}
A joint probability $\tilde{P}^\j$ in $\P^\j$ is {\bf strongly $\L$-conglome\-rable} if for every $F \in \A$ and $B \in \A_\L$ it holds
\begin{equation}
\pi(B)\inf_{H_i \subseteq B} \sigma(F|H_i) \le \tilde{P}^\j(F \wedge B) \le \pi(B)\sup_{H_i \subseteq B} \sigma(F|H_i).
\label{eq:cong}
\end{equation}
\end{definition}

Strong $\L$-conglomerability implies the classical notion of conglomerability introduced by de Finetti in \cite{dF-cong} in which condition (\ref{eq:cong}) is asked to hold only for $B = \Omega$. It is well-known that the notion of conglomerability due to Dubins \cite{dubins} is stronger than the one due to de Finetti and, under the assumption $\A_\L = \langle\L\rangle^*$, such notion is proved to be equivalent to $\L$-disintegrability of the joint probability on $\A$.

In \cite{brr-bayes} the authors analyse Dubins' notion of conglomerability without the assumption $\A_\L = \langle\L\rangle^*$ and show that $\L$-disintegrability for the joint probability is equivalent to its strong $\L$-conglomerability plus the S-integrability of $\sigma(F|\cdot)$ with respect to $\pi$, for $F \in \A$. Hence, under hypothesis $\A_\L = \langle\L\rangle^*$, Definition~\ref{def:cong} is equivalent to conglomerability in the sense of Dubins.

The notion of disintegrability introduced in previous section essentially relies in considering, for every $F \in \A$, $\sigma(F|\cdot)$ as a bounded function on $\L$ and in expressing the joint probability $P(F)$ as an average of $\sigma(F|\cdot)$ with respect to $\pi$. Hence, the S-integrability of $\sigma(F|\cdot)$ with respect to $\pi$ turns out to be fundamental.

Nevertheless, for $F \in \A$, a strategy $\sigma(F|\cdot)$ can always be considered as a (possibly not S-integrable with respect to $\pi$) bounded function on $\L$. In this case it is possible to define the lower and upper Stieltjes integrals
$$
\u{\int}\sigma(F|H_i)\pi(\d H_i)
\quad\mbox{and}\quad
\o{\int}\sigma(F|H_i)\pi(\d H_i).
$$

In particular, if $\sigma(F|\cdot)$ is S-integrable with respect to $\pi$ for every $F \in \A$, then Theorem~1.6 in \cite{brr-bayes} implies that
$$
P^\dd(F) = \int\sigma(F|H_i)\pi(\d H_i) = \u{\int}\sigma(F|H_i)\pi(\d H_i) = \o{\int}\sigma(F|H_i)\pi(\d H_i)
$$
and that $P^\dd$ is the unique strongly $\L$-conglomerable joint probability on $\A$ consistent with $\{\pi,\sigma\}$ which is also the unique $\L$-disintegrable one. This highlights that strong $\L$-conglomerability can be considered as a weakening of $\L$-disintegrability.

Our aim is to prove that in case $\sigma(F|\cdot)$ is not S-integrable with respect to $\pi$ for every $F \in \A$, i.e., necessarily $\A_\L \neq \langle\L\rangle^*$, then $\u{\int}\sigma(F|H_i)\pi(\d H_i)$ and $\o{\int}\sigma(F|H_i)\pi(\d H_i)$ bound the set of strongly $\L$-conglomerable joint probabilities on $\A$ consistent with $\{\pi,\sigma\}$.

First we show that the lower and upper S-integrals produce coherent values for the probability of each $F \in \A$.
\begin{proposition}
\label{prop:bounds-sc}
For every $F \in \A$ it holds
$$
\u{P}^\j(F) \le \u{\int} \sigma(F|H_i)\pi(\d H_i) \quad \mbox{and} \quad \o{\int} \sigma(F|H_i)\pi(\d H_i) \le \o{P}^\j(F).
$$
\end{proposition}

In general it holds $\u{P}^\j(F) < \u{\int} \sigma(F|H_i)\pi(\d H_i)$ and $\o{\int} \sigma(F|H_i)\pi(\d H_i) < \o{P}^\j(F)$: this is trivial when $\sigma(F|\cdot)$ is S-integrable with respect to $\pi$.

Next result provides a topological characterization of the set $\P^{\sc} \subseteq \P^\j$ of strongly $\L$-conglomerable
joint probabilities on $\A$ consistent with $\{\pi,\sigma\}$.

\begin{theorem}
\label{th:sc}
The set $\P^\sc$ is a non-empty compact subset of the space $[0,1]^\A$ endowed with the product topology of pointwise convergence and has envelopes $\u{P}^{\sc} = \min \P^\sc$ and $\o{P}^\sc = \max \P^\sc$ defined for every $F \in \A$ as
$$
\u{P}^\sc(F) = \u{\int} \sigma(F|H_i)\pi(\d H_i) \quad \mbox{and} \quad \o{P}^\sc(F) = \o{\int} \sigma(F|H_i)\pi(\d H_i).
$$
\end{theorem}

As shown in the proof of Theorem~\ref{th:sc}, the lower and upper S-integrals above coincide with the Choquet integrals of $\sigma(F|\cdot)$ computed, respectively, with respect to the inner $\pi_*$ and the outer measure $\pi^*$ induced by $\pi$ on $\langle\L\rangle^*$.

Except for the trivial case when $\sigma(F|\cdot)$ is S-integrable with respect to $\pi$, for $F \in \A$, it is well-known (see, e.g.,\cite{dF-cong,Seidenfeld}) that there are joint probabilities consistent with $\{\pi,\sigma\}$ that are not $\L$-conglomerable in the sense of de Finetti, thus they are neither strongly $\L$-conglomerable: this implies $\P^\sc \subseteq \P^\j$, in general. The following example shows that under particular choices of $\pi$ and $\sigma$, and the related Boolean algebras, it can happen $\P^\sc = \P^\j$.

\begin{example}
Let $\L = \{H_i\}_{i \in \NN}$ and $\E = \{E_1,E_2\}$ with $H_i \wedge E_j \neq \emptyset$, for every $i,j$.
Take $\A_\L = \langle\L\rangle$, $\A_\E = \langle\E\rangle$ and $\A = \langle\A_\L \cup \A_\E\rangle$, together with the
finitely additive prior probability defined for $K \in \A_\L$ as
$$
\pi(K) =
\left\{
\begin{array}{ll}
0 & \mbox{if $K = \bigvee_{i \in I} H_i$ and $\card I < \aleph_0$},\\
1 & \mbox{otherwise},
\end{array}
\right.
$$
and the statistical model on $\A_\E \times \L$ singled out for $i \in \NN$ by
$$
\lambda(E_1|H_i) =
\left\{
\begin{array}{ll}
1 & \mbox{if $i$ is even},\\
0 & \mbox{otherwise},
\end{array}
\right.
\quad\mbox{and}\quad \lambda(E_2|H_i) = 1 - \lambda(E_1|H_i).
$$
The statistical model $\lambda$ extends uniquely to a strategy $\sigma$ on $\A \times \L$.

Notice that $\sigma(E_1|H_i) = {\bf 1}_{A}(H_i)$ and $\sigma(E_2|H_i) = {\bf 1}_{A^c}(H_i)$ with $A = \bigvee_{i \in \NN}H_{2i}$, thus none of them is $S$-integrable with respect to $\pi$.

It holds
$$
\u{P}^\sc(E_1) = \u{\int}\sigma(E_1|H_i)\pi(\d H_i) = \intC{\bf 1}_A(H_i)\pi_*(\d H_i) = \pi_*(A) = 0,
$$
and an analogous computation shows $\u{P}^\sc(E_2) = 0$, thus $\o{P}^\sc(E_1) = 1 - \u{P}^\sc(E_2) = 1$. In turn, Proposition~\ref{prop:bounds-sc} implies that $\u{P}^\j(E_1) = 0$ and $\o{P}^\j(E_1) = 1$, so we obtain the same bounds for $E_1$ determined by the whole set of joint probabilities consistent with $\{\pi,\sigma\}$. Actually, simple computations show that every joint probability in $\P^\j$ is strongly $\L$-conglomerable, i.e., $\P^\sc = \P^\j$, so the envelopes (trivially) coincide on the whole $\A$.
\end{example}

As a natural consequence, the notion of full $\L$-disintegrability can be weakened in the following notion of full strong $\L$-conglomerability.

\begin{definition}
A full conditional probability $\tilde{Q}(\cdot|\cdot)$ on $\A$ extending $\{\pi,\sigma\}$ is {\bf fully strongly $\L$-conglomerable} if, denoting with $\tilde{\pi}^\p = \tilde{Q}_{|\A_\L \times \A_\L^0}$, for every $F|K \in \A \times \A_\L^0$ and every $B \in \A_\L$ such that $B \subseteq K$ it holds
\begin{equation}
\tilde{\pi}^\p(B|K)\inf_{H_i \subseteq B} \sigma(F|H_i) \le \tilde{Q}(F \wedge B|K) \le \tilde{\pi}^\p(B|K)\sup_{H_i \subseteq B} \sigma(F|H_i).
\label{eq:full-cong}
\end{equation}
\end{definition}

Thus we can restrict to the set
$$
\Q^\fsc = \{\tilde{Q} \,:\, \mbox{fully strongly $\L$-conglomerable f.c.p. on $\A$ extending $\{\pi,\sigma\}$}\},
$$
whose topological structure is considered in next theorem. Let us stress that in case $\sigma(F|\cdot)$ is $\A_\L$-continuous, for every $F \in \A$, then $\Q^{\fsc} = \Q^{\fd}$.

\begin{theorem}
\label{th:set-fsc}
The set $\Q^\fsc$ is a non-empty compact subset of the space $[0,1]^{\A \times \A^0}$ endowed with the product topology of pointwise convergence.
\end{theorem}

Concerning the lower envelope $\u{Q}^\fsc = \min \Q^\fsc$, it can be expressed as the minimum of full conditional probabilities extending $\{\pi,\sigma\}$ that are fully $\L$-disintegrable with respect to an extension of $\pi$ on $\langle\L\rangle^*$. At this aim, consider the core $\P_{\pi_*}$ induced by the inner measure $\pi_*$, which coincides with the set of all finitely additive probabilities extending $\pi$ on $\langle\L\rangle^*$.

Let $\B = \langle\A \cup \langle\L\rangle^*\rangle$ and $\rho$ be any strategy on $\B \times \L$ extending $\sigma$. For every $\tilde{\nu} \in \P_{\pi_*}$ we can consider the set $\Q^\fd_{\tilde{\nu}}$ of fully $\L$-disintegrable full conditional probabilities on $\B$ extending $\{\tilde{\nu},\rho\}$, whose lower envelope is denoted as $\u{Q}^\fd_{\tilde{\nu}} = \min \Q^\fd_{\tilde{\nu}}$. The following characterization of $\u{Q}^\fsc$ is then obtained.

\begin{theorem}
\label{th:lower-fsc}
For every $F|K \in \A \times \A^0$ it holds
$$
\u{Q}^\fsc(F|K) = \min\left\{\u{Q}^\fd_{\tilde{\nu}}(F|K) 	\,:\, \tilde{\nu} \in \P_{\pi^*}\right\}.
$$
\end{theorem}

\appendix
\section{Proofs for Section~3}
\label{appendix0}
\begin{proof}[Proof of Proposition~\ref{prop:uniqueness-strategy}]
Every $F \in \A$ is such that $F=\bigvee_{s=1}^m\bigwedge_{t=1}^{n_s} A_{s_t}$,
with $A_{s_t} \in \A_\L \cup \A_\E$. For $H_i \in \L$ it holds
$$
F \wedge H_i = \left(\bigvee_{s=1}^m\bigwedge_{t=1}^{n_s} A_{s_t}\right) \wedge H_i = \bigvee_{s=1}^m\left(\left(\bigwedge_{t=1}^{n_s} A_{s_t}\right) \wedge H_i\right).
$$

Define the index set
$
S = \left\{s \in \{1,\ldots,m\} \,:\, \left(\bigwedge_{t=1}^{n_s} A_{s_t}\right) \wedge H_i \neq \emptyset\right\},
$
and for each $s \in S$ define the index set
$
T_s = \left\{t \in \{1,\ldots,n_s\} \,:\, A_{s_t} \in \A_\E\right\}.
$
This implies that the event $F_{H_i} = \bigvee_{s \in S}\bigwedge_{t \in T_s} A_{s_t}$ belongs to $\A_\E$ and is such that
$F \wedge H_i = F_{H_i} \wedge H_i$, where $F_{H_i} = \emptyset$ if $S = \emptyset$ and $\bigwedge_{t \in T_s} A_{s_t} = \emptyset$ if $T_s = \emptyset$.

Let $\sigma$ be a strategy extending on $\A \times \L$ the statistical model $\lambda$ defined on $\A_\E \times \L$. For $F|H_i \in \A \times \L$ it must be
$
\sigma(F|H_i) = \sigma(F \wedge H_i|H_i) = \sigma(F_{H_i} \wedge H_i|H_i) = \sigma(F_{H_i}|H_i) = \lambda(F_{H_i}|H_i),
$
i.e., $\sigma$ is uniquely determined by $\lambda$. 
\end{proof}

\begin{proof}[Proof of Theorem~\ref{th:joint}]
The proof is trivial if $\L$ is finite. Thus suppose $\card \L \ge \aleph_0$, let $\G = (\A_\L \times \{\Omega\}) \cup (\A \times \L)$.
By Theorem~\ref{th:extension}, for every $F \in \A$, the interval of coherent extensions $\II_{F} = [\u{P}(F), \o{P}(F)]$ can be computed in terms of finite subfamilies of $\G$.

Since for every $\F_1 \subseteq \F_2 \subseteq \G$ with $\card \F_2 < \aleph_0$ one has $\u{P}^{\F_1}(F) \le \u{P}^{\F_2}(F)$, we can restrict to finite subfamilies of $\G$ containing a set of the form
$(\L^\F \times \{\Omega\}) \cup (\{F\} \times \{H_{i_h}\}_{h=1}^n )$,
where $\L^\F = \{H_{i_h}\}_{h=1}^n \cup \{B_k\}_{k=1}^t$ is a finite partition of $\Omega$ contained in $\A_\L$.
Indeed, every finite subfamily can be suitably enlarged in order to contain a set of this form.

For such a set $\F$ we have $\u{P}^\F(F) = \sum_{h=1}^n \sigma(F|H_{i_h})\pi(H_{i_h}) + \sum_{B_k \subseteq F} \pi(B_k)$ and so the thesis follows.
\end{proof}

\begin{proof}[Proof of Theorem~\ref{th:ext-prob-lik-inf}]
The statement is trivial if $F \wedge K = K$ since in this case $\tilde{P}(F|K) = 1$ for every $\tilde{P} \in \P$, for this suppose $F \wedge K \neq K$.

To prove condition {\it (i)}, suppose $\u{P}^\j(K) > 0$, which implies $\tilde{P}^\j(K) > 0$ for every $\tilde{P}^\j \in \P^\j$, and so $\u{P}(F|K) = \min\left\{\frac{\tilde{P}^\j(F \wedge K)}{\tilde{P}^\j(F \wedge K) + \tilde{P}^\j(F^c \wedge K)} \,:\, \tilde{P} \in \P\right\}$. The conclusion follows since the real function $\frac{x}{x+y}$ is increasing in $x$ and decreasing in $y$, so the minimum is attained in correspondence of $\frac{\u{P}^\j(F \wedge K)}{\u{P}^\j(F \wedge K) + U^\j(F^c,K)}$ or
$\frac{L^\j(F,K)}{L^\j(F,K) + \o{P}^\j(F^c \wedge K)}$.

To prove condition {\it (ii)}, let $\G = (\A_\L \times \{\Omega\}) \cup (\A \times \L)$ and assume $\u{P}^\j(K) = 0$.
By Theorem~\ref{th:extension}, for every $F|K \in \A \times \A^0$, the interval of coherent extensions $\II_{F|K} = [\u{P}(F|K), \o{P}(F|K)]$ can be computed in terms of finite subfamilies of $\G$.

Since for every $\F_1 \subseteq \F_2 \subseteq \G$ and $\card \F_2 < \aleph_0$ one has $\u{P}^{\F_1}(F|K) \le \u{P}^{\F_2}(F|K)$, for arbitrary $I_k' \subseteq I_k^{F|K}$ with $\card I_k' < \aleph_0$, $k=1,2,3$, and $I' = I_1' \cup I_2' \cup I_3'$,
we can restrict to finite subfamilies containing
$(\L^\F \times \{\Omega\}) \cup (\C_{\{F,K\}} \times \{H_i\}_{i \in I'})$,
where $\L^\F =\{H_i\}_{i \in I'} \cup \{B_k\}_{k=1}^t$ is a finite partition of $\Omega$ contained in $\A_\L$, and
$\C_{\{F,K\}} = \{A_h\}_{h=1}^m$, with $m \le 4$, is the set of atoms of the algebra generated by $\{F,K\}$.
Indeed, every finite subfamily can be suitably enlarged in order to contain a set of this form.

For such a finite subfamily $\F$, let $\C_\F = \{C_1,\ldots,C_q\}$ be the set of atoms of the algebra generated by $\L^\F \cup \C_{\{F,K\}}$.

Let $\C_1 = \{C_r \in \C_\F \,:\, \u{P}(C_r) = 0\}$.
As described in \cite{cs-libro} (see also \cite{cv-amai}) the lower bound
$\u{P}^\F(F|K)$ can be explicitly computed by solving the optimization problem with non-negative unknowns
$x_{r}^{1}$ for $C_r \in \C_{1}$,
$$
\minimize\left[\sum\limits_{C_r \subseteq {F \wedge K}} x_{r}^{1}\right]
$$
$$
\left\{
\begin{array}{ll}
x_{r}^{1} = \sigma(A_h|H_i) \cdot \left(\sum\limits_{C_s \subseteq H_i} x_{s}^1\right) & \mbox{if $\sigma(K|H_i) > 0$ and $\pi(H_i) = 0$}\\
& \mbox{and $i \in I'$ and $C_r = A_h \wedge H_i \in \C_1$},\\[2ex]
\sum\limits_{C_r \subseteq K} x_{r}^{1} = 1.
\end{array}
\right.
$$
Denote with ${\bf x}^1$, whose $r$-th component is ${\bf x}^1_r$, a solution of previous system.

If $I_2^{F|K} = \emptyset$, in order to be $F \wedge K \neq K$, it must be $I_3^{F|K} \neq \emptyset$ and so we can restrict to finite subfamilies having $I_3' \neq \emptyset$. In this case, previous system has always a solution such that $\sum\limits_{C_r \subseteq {F \wedge K}} {\bf x}_r^1 = 0$ and
$\sum\limits_{C_r \subseteq {F^c \wedge K}} {\bf x}_r^1 = 1$, which implies $\u{P}^\F(F|K)= 0$. Since every finite subfamily can be suitably enlarged to a finite subfamily having $I_3' \neq \emptyset$, then $\u{P}(F|K) = 0$.

If $I_2^{F|K} \neq \emptyset$ and $\card I_2^{F|K} \ge \aleph_0$ we can restrict to finite subfamilies having $I_2' \neq \emptyset$, for which previous system has always a solution such that $\sum\limits_{C_r \subseteq {F \wedge K}} {\bf x}_r^1 = 0$ and $\sum\limits_{C_r \subseteq {F^c \wedge K}} {\bf x}_r^1 = 1$, which implies $\u{P}^\F(F|K)= 0$. Since every finite subfamily can be suitably enlarged to a finite subfamily having $I_2' \neq \emptyset$, then $\u{P}(F|K) = 0$.

Finally, if $I_2^{F|K} \neq \emptyset$ and $\card I_2^{F|K} < \aleph_0$ we can restrict to finite subfamilies having $I_2' = I_2^{F|K}$ for which the minimum of previous optimization problem is achieved in correspondence of those solutions such that $\sum\limits_{C_r \subseteq K \wedge H_i} {\bf x}_{r}^1 = 1$ for $i \in I_2^{F|K}$, that implies $\sum\limits_{C_r \subseteq F \wedge K \wedge H_i} {\bf x}_{r}^1 = \frac{\sigma(F \wedge K|H_i)}{\sigma(K|H_i)}$ for $i \in I_2^{F|K}$, and then
$\u{P}^\F(F|K) = \min\limits_{i \in I_2^{F|K}} \frac{\sigma(F \wedge K|H_i)}{\sigma(K|H_i)}$. Since every finite subfamily can be suitably enlarged to a finite subfamily having $I_2' = I_2^{F|K}$, then $\u{P}(F|K) = \min\limits_{i \in I_2^{F|K}} \frac{\sigma(F \wedge K|H_i)}{\sigma(K|H_i)}$.
\end{proof}

\section{Proofs for Section~4}
\label{appendix1}
Throughout this section $\sigma(F|\cdot)$, viewed as a function of the second variable, is assumed to be an $\A_\L$-continuous function defined on $\L$, for every $F \in \A$.

\begin{proof}[Proof of Theorem~\ref{th:prior-lik-inf-dis}]
The proof is trivial if $F \wedge K = K$ or $P^\dd(K) > 0$, thus suppose $F \wedge K \neq K$ and $P^\dd(K) = 0$.

The proof of condition {\it (ii)} follows the same line of that of condition {\it (ii)} of Theorem~\ref{th:ext-prob-lik-inf}. Let $\G = \A \times (\{\Omega\} \cup \L)$ and take arbitrary $I_k' \subseteq I_k^{F|K}$ with $\card I_k' < \aleph_0$, $k=1,2,3$, and $I' = I_1' \cup I_2' \cup I_3'$. Consider a finite subfamily $\F \subseteq \G$ containing $(\L^\F \times \{\Omega\}) \cup (\C_{\{F,K\}} \times \{H_i\}_{i \in I'})$,
where $\L^\F =\{H_i\}_{i \in I'} \cup \{B_k\}_{k=1}^t$ is a finite partition of $\Omega$ contained in $\A_\L$, and
$\C_{\{F,K\}} = \{A_h\}_{h=1}^m$, with $m \le 4$, is the set of atoms of the algebra generated by $\{F,K\}$.
Let $\C_\F = \{C_1,\ldots,C_q\}$ be the set of atoms of the algebra generated by $\L^\F \cup \C_{\{F,K\}}$ and $\C_1 = \{C_r \in \C_\F \,:\, P^\d(C_r) = 0\}$. Then, the conclusion follows solving an optimization problem analogous to the one in the proof of condition {\it (ii)} of Theorem~\ref{th:ext-prob-lik-inf}.
\end{proof}

Note that full $\L$-disintegrability is essentially determined by the set of full conditional prior probabilities $\P^\p =  \{\tilde{\pi}^\p = \tilde{P}_{|\A_\L \times \A_\L^0} \,:\, \tilde{P} \in \P\}$, with $\P$ the set of f.c.p. on $\A$ extending $\{\pi,\sigma\}$, whose lower envelope $\u{\pi}^\p = \min \P^\p$ is characterized in the following corollary that is an immediate consequence of Theorem~\ref{th:prior-lik-inf-dis}.

\begin{remark}
\label{rem:no-strategy}
The same class $\P^\p$ is obtained extending coherently the sole $\pi$ to $\A_\L \times \A_\L^0$, i.e., the strategy $\sigma$ does not affect it. Indeed, for every $H_i \in \L$, every $\tilde{\pi}^\p \in \P^\p$ is only asked to satisfy $\tilde{\pi}^\p_{|\A_\L \times \{H_i\}} = \sigma_{|\A_\L \times \{H_i\}}$ which constitutes a vacuous constraint.
\end{remark}

\begin{corollary}
\label{cor:cond-prior}
The lower envelope $\u{\pi}^\p$ of the set $\P^\p$ of coherent extensions of $\{\pi,\sigma\}$ to $\A_\L \times \A_\L^0$ satisfies the following properties:
\begin{itemize}
\item[\it (i)] $\u{\pi}^\p(\cdot|K)$ is totally monotone on $\A_\L$, for every $K \in \A_\L^0$;
\item[\it (ii)] for every $F|K \in \A_\L \times \A_\L^0$ it holds $\u{\pi}^\p(F|K) = 1$ when $F \wedge K = K$ and if
$F \wedge K \neq K$
$$
\u{\pi}^\p(F|K) =
\left\{
\begin{array}{ll}
\frac{\pi(F \wedge K)}{\pi(K)} & \mbox{if $\pi(K) > 0$},\\[2ex]
0 & \mbox{otherwise}.
\end{array}
\right.
$$
\end{itemize}
\end{corollary}

Previous result implies that, for $K \in \A_\L^0$, $\u{\pi}^\p(\cdot|K)$ is a finitely additive probability if $\pi(K) > 0$, and otherwise it is a totally monotone capacity {\it vacuous at} $K$ (i.e., for every $F \in \A_\L$ it holds $\u{\pi}^\p(F|K) = 1$ if $K \subseteq F$ and $0$ otherwise).

Given $\tilde{\pi}^\p$ in $\P^\p$, for every $F|K \in \A \times \A_\L^0$ define the function
\begin{equation}
\tilde{P}^\fd(F|K) = \int \sigma(F|H_i)\tilde{\pi}^\p(\d H_i|K).
\label{eq:full-dis}
\end{equation}
Next proposition shows that the function $\tilde{P}^\fd$ is a conditional probability on $\A \times \A_\L^0$ extending $\{P^\dd,\sigma,\tilde{\pi}^\p\}$ which will be referred to as {\it fully $\L$-disintegrable extension} in the following.
It is easily seen that any full conditional probability on $\A$ further extending $\tilde{P}^\fd$ will, in turn, be fully $\L$-disintegrable.

\begin{proposition}
\label{prop:fully-L-dis}
The function $\tilde{P}^\fd$ defined as in equation (\ref{eq:full-dis}) is a conditional probability on $\A \times \A_\L^0$ extending
$\{P^\dd,\sigma,\tilde{\pi}^\p\}$.
\end{proposition}
\begin{proof}[Proof of Proposition~\ref{prop:fully-L-dis}]
The properties of the Stieltjes integral \cite{bhaskara} immediately imply that $\tilde{P}^\fd$ extends $\{P^\dd,\sigma,\tilde{\pi}^\p\}$.
We show that $\tilde{P}^\fd$ is a conditional probability on $\A \times \A_\L^0$. Conditions {\bf (C1)} and {\bf (C2)} follow by properties {\bf (S1)} and {\bf (S2)} of $\sigma$ and linearity of the Stieltjes integral. Finally, condition {\bf (C3)} follows since  for every $K, E\wedge K\in \A_\L^0$ and $E, F\in\A$, the equation
$\tilde{P}^\fd(E\wedge F|K)= \tilde{P}^\fd(E|K) \cdot \tilde{P}^\fd(F|E\wedge K)$ is trivially satisfied if $\tilde{P}^\fd(E|K) = 0$, while if
$\tilde{P}^\fd(E|K) > 0$ it holds
\begin{eqnarray*}
\frac{\tilde{P}^\fd(E \wedge F|K)}{\tilde{P}^\fd(E|K)} &=& \frac{1}{\tilde{P}^\fd(E|K)} \int \sigma(E \wedge F|H_i)\tilde{\pi}^\p(\d H_i|K)\\
&=& \int \sigma(E \wedge F|H_i)\frac{\tilde{\pi}^\p(\d H_i|K)}{\tilde{P}^\fd(E|K)}\\
&=& \int \sigma(E \wedge F|H_i)\tilde{\pi}^\p(\d H_i|E \wedge K) = \tilde{P}^\fd(F|E \wedge K).
\end{eqnarray*}
\end{proof}

The fully $\L$-disintegrable extension $\tilde{P}^\fd$ can then be further extended through coherence to a full conditional probability $\tilde{Q}$ on $\A$. The extension $\tilde{Q}$ is generally not unique so we have a set
$$
\Q_{\tilde{P}^\fd} = \left\{\tilde{Q}\,:\, \mbox{f.c.p on $\A$ extending $\tilde{P}^\fd$}\right\}
$$
whose lower envelope is $\u{Q}_{\tilde{P}^\fd} = \min \Q_{\tilde{P}^\fd}$.

\begin{lemma}
\label{lem:full-fully-L-dis}
The lower envelope $\u{Q}_{\tilde{P}^\fd}(\cdot|\cdot)$ is such that for every
$F|K \in \A \times \A^0$ it holds $\u{Q}_{\tilde{P}^\fd} (F|K) = 1$ when $F \wedge K = K$ and if $F \wedge K \neq K$
$$
\u{Q}_{\tilde{P}^\fd} (F|K) =
\left\{
\begin{array}{ll}
\frac{\tilde{P}^\fd(F \wedge K|A)}{\tilde{P}^\fd(K|A)} & \mbox{if $\exists \, A \in \A_\L^0$ s.t. $K \subseteq A$ and $\tilde{P}^\fd(K|A) > 0$},\\[2ex]
0 & \mbox{otherwise}.
\end{array}
\right.
$$
\end{lemma}
\begin{proof}[Proof of Lemma~\ref{lem:full-fully-L-dis}]
The proof is trivial in case $F \wedge K = K$ or there exists $A \in \A_\L^0$ s.t. $K \subseteq A$ and $\tilde{P}^\fd(K|A) > 0$, thus suppose $F \wedge K \neq K$ and $\tilde{P}^\fd(K|A) = 0$ for every $A \in \A_\L^0$ with $K \subseteq A$.

Under this hypothesis, let $\G = \A \times \A_\L^0$. By Theorem~\ref{th:extension}, for every $F|K \in \A \times \A^0$, the interval of coherent extensions $\II_{F|K} = \left[\u{Q}_{\tilde{P}^\fd} (F|K), \o{Q}_{\tilde{P}^\fd} (F|K)\right]$ can be computed in terms of finite subfamilies of $\G$.

Since for every $\F_1 \subseteq \F_2 \subseteq \G$ and $\card \F_2 < \aleph_0$ one has $\u{Q}_{\tilde{P}^\fd}^{\F_1}(F|K) \le \u{Q}_{\tilde{P}^\fd}^{\F_2}(F|K)$, we can restrict to finite subfamilies of $\G$ of the form $\F = \B \times \B_\L^0$ where $\B \subseteq \A$ and $\B_\L \subseteq \A_\L$ are finite Boolean algebras with $\{F,K\} \subseteq \B$. Indeed, every finite subfamily can be suitably enlarged in order to meet this form. Now, Theorem~4 in \cite{cpv-ins2014}
implies that $\u{Q}_{\tilde{P}^\fd}^\F(F|K) = 0$ and since this holds for every finite subfamily $\F$ the proof follows.
\end{proof}

Next lemma investigates the topology of the set $\P^\fd$ of fully $\L$-disintegrable extensions of $\{\pi,\sigma\}$ on $\A \times \A_\L^0$.

\begin{lemma}
\label{lem:compact-fully-dis}
The set $\P^\fd$ is a non-empty compact subset of the space $[0,1]^{\A \times \A_\L^0}$ endowed with the product topology of pointwise convergence.
\end{lemma}
\begin{proof}[Proof of Lemma~\ref{lem:compact-fully-dis}]
The coherence of $\{\pi,\sigma\}$ implies that the set $\P^\p$ of coherent extensions to $\A_\L \times \A_\L^0$ is not empty, moreover, for every $\tilde{\pi}^\p \in \P^\p$, the function $\tilde{P}^\fd$ defined as in equation (\ref{eq:full-dis})
is an element of $\P^\fd$, thus it is not empty.

We prove the compactness of $\P^\fd$. By Thychonoff's theorem $[0,1]^{\A \times \A_\L^0}$ is a compact space endowed with the product topology of pointwise convergence, moreover, it is Hausdorff. Hence, it is sufficient to prove that $\P^\fd$ is a closed subset of $[0,1]^{\A \times \A_\L^0}$ as closed subsets of a compact space are compact.


Thus let $(\tilde{P}^\fd_\alpha)_\alpha$ be a net of elements of $\P^\fd$ converging pointwise to $\tilde{P}^\fd$ and denote with $(\tilde{\pi}^\p_\alpha)_\alpha$ and $\tilde{\pi}^\p$, respectively, the corresponding restrictions on $\A_\L \times \A_\L^0$.

Theorem~5 in \cite{regazzini} implies that $\tilde{P}^\fd$ is a conditional probability extending $\{\pi,\sigma\}$, thus,
it remains to prove that $\tilde{P}^\fd$ is a fully $\L$-disintegrable extension of $\{\pi,\sigma\}$. For every $\alpha$ and every $E|H \in \A \times \A_\L^0$ we have
$$
\tilde{P}^\fd_\alpha(E|H) = \int \sigma(E|H_i)\tilde{\pi}^\p_\alpha(\d H_i | H),
$$
moreover, since for every fixed $H \in \A_\L^0$ the net of finitely additive probabilities $(\tilde{\pi}^\p_\alpha(\cdot|H))_\alpha$ on $\A_\L$ converges pointwise to the finitely additive probability $\tilde{\pi}^\p(\cdot|H)$ on $\A_\L$, Theorem~3.6 in \cite{girotto} implies $\int \sigma(E|H_i) \tilde{\pi}^\p_\alpha(\d H_i | H) \longrightarrow \int \sigma(E|H_i) \tilde{\pi}^\p(\d H_i | H)$ that is $\tilde{P}^\fd_\alpha(E|H) \longrightarrow \int \sigma(E|H_i) \tilde{\pi}^\p(\d H_i | H)$ and so we have $\tilde{P}^\fd(E|H) = \int \sigma(E|H_i) \tilde{\pi}^\p(\d H_i | H)$.
\end{proof}

The following lemma characterizes the lower envelope of the set $\u{P}^\fd = \min \P^\fd$.

\begin{lemma}
\label{lem:lower-fully-L-dis}
The lower envelope $\u{P}^\fd(\cdot|\cdot)$ is such that for every
$F|K \in \A \times \A_\L^0$ it holds $\u{P}^\fd(F|K) = 1$ when $F \wedge K = K$ and if $F \wedge K \neq K$
$$
\u{P}^\fd(F|K) = \intC \sigma(F|H_i) \u{\pi}^\p(\d H_i|K) =
\left\{
\begin{array}{ll}
\frac{\int \sigma(F \wedge K|H_i)\pi(\d H_i)}{\pi(K)} & \mbox{if $\pi(K) > 0$},\\[2ex]
\inf\limits_{H_i \subseteq K}\sigma(F|H_i) & \mbox{otherwise}.
\end{array}
\right.
$$
\end{lemma}
\begin{proof}[Proof of Lemma~\ref{lem:lower-fully-L-dis}]
The proof immediately follows by Corollary~\ref{cor:cond-prior} and the properties of the Choquet integral \cite{denneberg}.
\end{proof}

Finally, the set of fully $\L$-disintegrable full conditional probabilities on $\A$ extending $\{\pi,\sigma\}$ is given by
\begin{eqnarray*}
\Q^\fd &=& \{\tilde{Q} \,:\, \mbox{fully $\L$-disintegrable f.c.p. on $\A$ extending $\{\pi,\sigma\}$}\}\\
&=& \bigcup \{\Q_{\tilde{P}^\fd} \,:\, \tilde{P}^\fd \in \P^\fd\}.\nonumber
\end{eqnarray*}

The set $\Q^\fd$ turns out to be a compact subset of the space $[0,1]^{\A \times \A^0}$ endowed with the product topology of pointwise convergence.

\begin{proof}[Proof of Theorem~\ref{th:compact}]
The proof follows by the proof of Lemma~\ref{lem:compact-fully-dis} and the closure of coherent conditional probabilities under limits of nets \cite{regazzini}.
\end{proof}

Next theorem characterizes the lower envelope $\u{Q}^\fd$ of the set $\Q^\fd$, relying on the following functions, defined for $F \in \A$, $K \in \A^0$ and $A \in \A_\L^0$ with $K \subseteq A$ as
\begin{eqnarray}
L^\fd(F, K;A) &=& \min\left\{\tilde{P}^\fd(F \wedge K|A) \,:\, \tilde{P}^\fd \in \P^\fd,\right.\label{eq:Ld}\\\nonumber
&&~~~~~~~~~~~~\left.\tilde{P}^\fd(F^c \wedge K|A) = \o{P}^\fd(F^c \wedge K|A)\right\},\\
U^\fd(F, K;A) &=& \max\left\{\tilde{P}^\fd(F \wedge K|A) \,:\, \tilde{P}^\fd \in \P^\fd,\right.\label{eq:Ud}\\\nonumber
&&~~~~~~~~~~~~\left.\tilde{P}^\fd(F^c \wedge K|A) = \u{P}^\fd(F^c \wedge K|A)\right\},
\end{eqnarray}
for which it holds
$$
\u{P}^\fd(F \wedge K|A) \le L^\fd(F, K;A) \le U^\fd(F, K;A) \le \o{P}^\fd(F \wedge K|A).
$$

\begin{remark}
\label{rem:compact}
Notice that in (\ref{eq:Ld}) we can write $\min$ in place of $\inf$ since the set
$
\D = \left\{\tilde{P}^\fd \in \P^\fd \,:\, \tilde{P}^\fd(F^c \wedge K|A) = \o{P}^\fd(F^c \wedge K|A)\right\}
$
is a closed (and so compact) subset of $\P^\fd$ endowed with the relative topology inherited by $[0,1]^{\A \times \A_\L^{0}}$. Indeed, every converging net $(\tilde{P}^\fd_\alpha)_\alpha$ contained in $\D$ has limit $\tilde{P}^\fd$ contained in $\D$ since $\tilde{P}^\fd$ belongs to $\P^\fd$ and $\lim_\alpha \tilde{P}^\fd_\alpha(F^c \wedge K|A) =
\lim_\alpha \o{P}^\fd(F^c \wedge K|A) = \o{P}^\fd(F^c \wedge K|A)$. Analogous consideration holds for equation (\ref{eq:Ud}).
\end{remark}

In order to prove the following lemma it is useful to recall that a coherent conditional probability $P$ on a set $\G$ of conditional events is monotone with respect to the following implication among conditional events \cite{gn}
\begin{equation}
E|H \gn F|K \Longleftrightarrow
\left\{
\begin{array}{ll}
E \wedge H \subseteq F \wedge K,\\
E^c \wedge H \supseteq F^c \wedge K,
\end{array}
\right.
\end{equation}
that is $E|H \gn F|K \Longrightarrow P(E|H) \le P(F|K)$, for $E|H,F|K \in \G$.

\begin{lemma}
\label{lem:lower-fully-L-dis-full}
The lower envelope $\u{Q}^\fd(\cdot|\cdot)$ is such that
for every $F|K \in \A \times \A^0$ it holds $\u{Q}^\fd(F|K) = 1$ when $F \wedge K = K$ and if $F \wedge K \neq K$, then
if there exists $A \in \A_\L^0$ such that $K \subseteq A$ and $\u{P}^\fd(K|A) > 0$ we have that
\begin{eqnarray*}
\u{Q}^\fd(F|K) &=& \min\left\{
\frac{\u{P}^\fd(F \wedge K|A)}{\u{P}^\fd(F \wedge K|A) + U^\fd(F^c,K;A)},\right.\\
&&~~~~~~~~~~~~\left.\frac{L^\fd(F,K;A)}{L^\fd(F,K;A) + \o{P}^\fd(F^c \wedge K|A)}
\right\},
\end{eqnarray*}
otherwise $\u{Q}^\fd(F|K) = 0$.
\end{lemma}
\begin{proof}[Proof of Lemma~\ref{lem:lower-fully-L-dis-full}]
The statement is trivial if $F \wedge K = K$, thus suppose $F \wedge K \neq K$.

If there exists $A \in \A_\L^0$ with $K \subseteq A$ such that $\u{P}^\fd(K|A) > 0$, then it holds $\tilde{P}^\fd(K|A) > 0$ for every $\tilde{P}^\fd \in \P^\fd$, and so we have 
$$
\u{Q}^\fd(F|K) = \min\left\{\frac{\tilde{P}^\fd(F \wedge K|A)}{\tilde{P}^\fd(F \wedge K|A) + \tilde{P}^\fd(F^c \wedge K|A)} \,:\, \tilde{P}^\fd \in \P^\fd\right\}.
$$
The conclusion follows since the real function $\frac{x}{x+y}$ is increasing in $x$ and decreasing in $y$, so
the minimum is attained in correspondence of
$\frac{\u{P}^\fd(F \wedge K|A)}{\u{P}^\fd(F \wedge K|A) + U^\fd(F^c,K;A)}$ or
$\frac{L^\fd(F,K;A)}{L^\fd(F,K;A) + \o{P}^\fd(F^c \wedge K|A)}$.

Otherwise, for all $A \in \A_\L^0$ with $K \subseteq A$ it holds $\u{P}^\fd(K|A) = 0$, which implies for every such $A$ the existence of
$\tilde{P}^\fd_A \in \P^\fd$ such that $\tilde{P}^\fd_A(K|A) = 0$ and so $\tilde{P}^\fd_A(K|B) = 0$ for every $B \in \A_\L^0$ with
$A \subseteq B$.

We show the existence of $\tilde{P}^\fd_0 \in \P^\fd$ such that $\tilde{P}^\fd_0(K|A) = 0$ for all $A \in \A_\L^0$ with $K \subseteq A$. By Lemma~\ref{lem:compact-fully-dis}, $\P^\fd$ is a compact subset of $[0,1]^{\A \times \A_\L^0}$ endowed with the product topology of pointwise convergence, thus $\P^\fd$ is a compact space with the relative topology inherited by $[0,1]^{\A \times \A_\L^0}$. In turn, the compactness of $\P^\fd$ is equivalent to the fact that every family of
non-empty closed subsets of $\P^\fd$ with the finite intersection property has non-empty intersection.

For an arbitrary finite subalgebra $\B_\L \subseteq \A_\L$ define
$K^*_{\B_\L} = \bigwedge \{B \in \B_\L^0 \,:\, K \subseteq B\}$, which belongs to $\B_\L^0$ since $\B_\L$ is finite.
Introduce the collection
$$
\EE_0 = \left\{\D_0^{\B_\L} = \left\{\tilde{P}^\fd \in \P^\fd \,:\, \tilde{P}^\fd(K|K_{\B_\L}^*) = 0\right\} \,:\, \mbox{$\B_\L \subseteq \A_\L$, $\card \B_\L < \aleph_0$}\right\},
$$
which is easily seen to be a family of non-empty closed subsets of $\P^\fd$.

We show that $\EE_0$ has the finite intersection property.
For any ${\B_\L}_1,\ldots,{\B_\L}_n$ finite subalgebras of $\A_\L$, the corresponding generated Boolean algebra $\B_\L' = \langle\bigcup_{i=1}^n {\B_\L}_i\rangle$ is still a finite subalgebra of $\A_\L$, moreover,
$K_{\B_\L'}^* \subseteq K_{{\B_\L}_i}^*$ for $i=1,\ldots,n$, and so $K|K_{{\B_\L}_i}^* \gn K|K_{\B_\L'}^*$ for $i=1,\ldots,n$. Hence, for every $\tilde{P}^\fd \in \D_0^{\B_\L'}$ we have $\tilde{P}^\fd(K|K_{\B_\L'}^*) = 0$ and by the monotonicity of $\tilde{P}^\fd$ with respect to $\gn$ relation, it follows $\tilde{P}^\fd(K|K_{{\B_\L}_i}^*) = 0$ for $i=1,\ldots,n$, and so $\tilde{P}^\fd \in \D_0^{{\B_\L}_i}$ for $i=1,\ldots,n$.
This implies $\bigcap_{i=1}^n \D_0^{{\B_\L}_i} \neq \emptyset$ and so $\EE_0$ satisfies the finite intersection property which, in turn, implies $\bigcap \EE_0 \neq \emptyset$, i.e., there exists $\tilde{P}^\fd_0 \in \bigcap \EE_0$ such that $\tilde{P}^\fd_0(K|A) = 0$ for every $A \in \A_\L^0$ with $K \subseteq A$.

Finally, Lemma~\ref{lem:full-fully-L-dis} implies $\u{Q}^\fd(F|K) = \u{Q}_{\tilde{P}^\fd_0}(F|K) = 0$.
\end{proof}

\begin{proof}[Proof of Theorem~\ref{th:fully-L-dis-fcp}]
The proof follows by Lemma~\ref{lem:lower-fully-L-dis} and Lemma~\ref{lem:lower-fully-L-dis-full}.
\end{proof}

\section{Proofs for Section~5}
\label{appendix2}

\begin{proof}[Proof of Proposition~\ref{prop:bounds-sc}]
We prove only the first inequality as the other has similar proof. By Theorem~\ref{th:joint}, for every finite partition $\L^\F = \{A_u\}_{u=1}^m = \{H_{i_h}\}_{h=1}^n \cup \{B_k\}_{k=1}^t$ such that $m = n + t$ and $\L^\F \subseteq \A_\L$ it holds
\begin{eqnarray*}
&&\sum_{h=1}^n \sigma(F|H_{i_h})\pi(H_{i_h}) + \sum_{B_k \subseteq F}\pi(B_k)\\
&&=\sum_{h=1}^n \left(\inf_{H_i \subseteq H_{i_h}}\sigma(F|H_i)\right)\pi(H_{i_h}) + \sum_{B_k \subseteq F}\left(\inf_{H_i \subseteq B_k}\sigma(F|H_i)\right)\pi(B_k)\\
&&\le \sum_{h=1}^n \left(\inf_{H_i \subseteq H_{i_h}}\sigma(F|H_i)\right)\pi(H_{i_h}) + \sum_{k=1}^t\left(\inf_{H_i \subseteq B_k}\sigma(F|H_i)\right)\pi(B_k).
\end{eqnarray*}
Hence, since both $\u{P}^\j(F)$ and $\u{\int}\sigma(F|H_i)\pi(\d H_i)$ are computed taking the supremum over all the finite partitions $\L^\F \subseteq \A_\L$ the conclusion follows.
\end{proof}

\begin{proof}[Proof of Theorem~\ref{th:sc}]
First we prove $\P^\sc$ is not empty. The prior probability $\pi$ can be extended to $\langle\L\rangle^*$ giving rise to a set of finitely additive probabilities $\W = \{\tilde{\nu(\cdot)}\}$
whose lower and upper envelopes $\u{\nu} = \min \W$ and $\o{\nu} = \max \W$ are, respectively, a totally monotone and a totally alternating capacity (see, e.g., \cite{cpv-sma,denneberg-cond,choquet}) and coincide with the inner and outer measure induced by $\pi$ \cite{bhaskara}, i.e., $\W = \P_{\pi_*}$. For every $\tilde{\nu} \in \W$, define the function $\tilde{P}$ setting for every $F \in \A$
$$
\tilde{P}(F) = \int\sigma(F|H_i)\tilde{\nu}(\d H_i),
$$
which is seen to be a finitely additive probability on $\A$ consistent with $\{\pi,\sigma\}$.
We show that $\tilde{P}$ is strongly $\L$-conglomerable. Let $\B = \langle\A \cup \langle\L\rangle^*\rangle$ and $\rho$ be any strategy on $\B \times \L$ extending $\sigma$. The assessment $\{\tilde{\nu},\rho\}$ is coherent and the joint probability $\tilde{R}$ on $\B$ setting for every $A \in \B$
$$
\tilde{R}(A) = \int \rho(A|H_i)\tilde{\nu}(\d H_i),
$$
is consistent with $\{\tilde{\nu},\rho\}$ and is an extension of $\tilde{P}$.
Since for every $A \in \B$, $\rho(A|\cdot)$ is S-integrable with respect to $\tilde{\nu}$, Theorem~1.6 in \cite{brr-bayes} implies that for every $A \in \B$ and $D \in \langle\L\rangle^*$ it holds
$$
\tilde{\nu}(D) \inf_{H_i \subseteq D} \rho(A|H_i) \le \tilde{R}(A \wedge D) \le \tilde{\nu}(D) \sup_{H_i \subseteq D} \rho(A|H_i)
$$
which for every $F \in \A$ and $B \in \A_\L$ reduces to
$$
\pi(B) \inf_{H_i \subseteq B} \sigma(F|H_i) \le \tilde{P}(F \wedge B) \le \pi(B) \sup_{H_i \subseteq B} \sigma(F|H_i).
$$
This implies that the set $\P^\sc = \{\tilde{P} \,:\, \tilde{\nu} \in \W\}$ is not empty.

To prove $\P^\sc$ is compact, it is sufficient to consider a net $(\tilde{P}_\alpha)_\alpha$ in $\P^\sc$ converging pointwise to $\tilde{P}$. The compactness of $\P^\j$ implies $\tilde{P}$ is an element of $\P^\j$, moreover, since the pointwise limits of nets preserve non-strict inequalities, it follows that $\tilde{P}$ is also an element of $\P^\sc$ and the claim follows.

By Proposition~3 in \cite{schmeidler}, for every $F \in \A$ the lower and upper envelopes of the set $\P^\sc$ are given by
\begin{eqnarray*}
\u{P}^\sc(F) &=& 
\intC \sigma(F|H_i)\u{\nu}(\d H_i) = \u{\int}\sigma(F|H_i) \pi(\d H_i),\\
\o{P}^\sc(F) &=& 
\intC \sigma(F|H_i)\o{\nu}(\d H_i) = \o{\int}\sigma(F|H_i) \pi(\d H_i),
\end{eqnarray*}
where the last equality follows in both equations by Theorem~8.25 in \cite{deCooman-book}.
\end{proof}

The proof of Theorem~\ref{th:set-fsc} relies on the following lemma.

\begin{lemma}
\label{lem:cond-inner}
Let $\A$ be a Boolean algebra, $P$ a full conditional probability on $\A$, and $\P = \{\tilde{P}(\cdot|\cdot)\}$ the set of all conditional probabilities extending $P$ on $\langle\A\rangle^* \times \A^0$. The lower envelope $\u{P} = \min \P$ is such that for every $K \in \A^0$, $\u{P}(\cdot|K)$ coincides with the inner measure on $\langle\A\rangle^*$ generated by $P(\cdot|K)$, thus is a totally monotone capacity.
\end{lemma}
\begin{proof}[Proof of Lemma~\ref{lem:cond-inner}]
For every $F|K \in \langle\A\rangle^* \times \A^0$, Theorem~\ref{th:extension} implies that we can restrict to finite subfamilies $\F \subseteq \A \times \A^0$ of the form $\F = \B \times \B^0$, with $\B \subseteq \A$ finite subalgebra containing $K$. Let us denote with $B$ the maximal element of $\B$ with respect to implication relation such that $B \subseteq F$.
In turn, this implies
$$
\u{P}(F|K) = \sup\{P(B|K) \,:\, B \subseteq F, B \in \A\},
$$
so $\u{P}(\cdot|K)$ coincides with the inner measure on $\langle\A\rangle^*$ generated by $P(\cdot|K)$ and is therefore a totally monotone capacity.
\end{proof}

\begin{proof}[Proof of Theorem~\ref{th:set-fsc}]
We prove first that $\Q^\fsc$ is not empty. At this aim, consider the set $\P^\p = \{\tilde{\pi}^\p(\cdot|\cdot)\}$ of conditional prior probabilities full on $\A_\L$ extending $\{\pi,\sigma\}$. Remark~\ref{rem:no-strategy} implies that every $\tilde{\pi}^\p$ in $\P^\p$ can be coherently extended to $\langle\L\rangle^* \times \A_\L^0$ without being affected from $\sigma$, obtaining a set $\W^\p = \{\tilde{\nu}^\p(\cdot|\cdot)\}$ of conditional probabilities on $\langle\L\rangle^* \times \A_\L^0$ with lower and upper envelopes $\u{\nu}^\p = \min \W^\p$ and $\o{\nu}^\p = \max \W^\p$.
By Lemma~\ref{lem:cond-inner}, for every $K \in \A_\L^0$, $\u{\nu}^\p(\cdot|K)$ is a totally monotone capacity on $\langle\L\rangle^*$.

Thus, for a fixed $\tilde{\pi}^\p$ in $\P^\p$, for every $K \in \A_\L^0$ the proof goes along the same line of the proof of Thereom~\ref{th:sc} using $\tilde{\pi}^\p(\cdot|K)$ in place of $\pi(\cdot)$, and the claim follows.

To prove $\Q^\fsc$ is compact, it is sufficient to consider a net $(\tilde{Q}_\alpha)_\alpha$ in $\Q^\fsc$ converging pointwise to $\tilde{Q}$. Denote with $(\tilde{\pi}^\p_\alpha)_\alpha$ and ${\tilde{\pi}^\p}$ the restrictions of $(\tilde{Q}_\alpha)_\alpha$ and $\tilde{Q}$ on $\A_\L \times \A_\L^0$, respectively.
Theorem~5 in \cite{regazzini} implies that $\tilde{Q}$ is a full conditional probability on $\A$ extending $\{\pi,\sigma\}$.
For every $\alpha$, $F|K \in \A \times \A_\L^0$ and $B \in \A_\L$ it holds
$$
\tilde{\pi}^\p_\alpha(B|K)\inf_{H_i \subseteq B \wedge K} \sigma(F|H_i) \le \tilde{Q}_\alpha(F \wedge B|K) \le \tilde{\pi}^\p_\alpha(B|K)\sup_{H_i \subseteq B \wedge K} \sigma(F|H_i),
$$
and since the pointwise limits of nets preserve non-strict inequalities, it follows
$$
\tilde{\pi}^\p(B|K)\inf_{H_i \subseteq B \wedge K} \sigma(F|H_i) \le \tilde{Q}(F \wedge B|K) \le \tilde{\pi}^\p(B|K)\sup_{H_i \subseteq B \wedge K} \sigma(F|H_i),
$$
that is, $\tilde{Q}$ is an element of $\Q^\fsc$ and the claim follows.
\end{proof}

\begin{proof}[Proof of Theorem~\ref{th:lower-fsc}]
The proof easily follows by the proof of Theorem~\ref{th:set-fsc}.
\end{proof}

\section*{Acknowledgements}
\noindent This work was partially supported by INdAM-GNAMPA through the Project 2015 {\tt U2015/000418} {\it ``Envelopes of coherent conditional probabilities and their applications to statistics and artificial intelligence''}, and by the Italian Ministry of Education, University and
Research, funding of Research Projects of National Interest (PRIN 2010-11) under the grant {\tt 2010FP79LR\_003}
{\it ``Logical methods of information management''}.

\bibliographystyle{plain}
\bibliography{biblio.bib}














\end{document}